\documentclass[11pt]{article}

\usepackage[a4paper,margin=25mm]{geometry}
\usepackage{authblk}
\usepackage[T1]{fontenc}
\usepackage[utf8]{inputenc}
\usepackage{lmodern}
\usepackage{amsmath,amssymb,amsthm,mathtools}
\usepackage{array,longtable,booktabs}
\newcolumntype{P}[1]{>{\raggedright\arraybackslash}p{#1}}
\usepackage{microtype}
\usepackage[dvipsnames]{xcolor}
\usepackage[colorlinks=true,linkcolor=MidnightBlue,citecolor=MidnightBlue,urlcolor=MidnightBlue]{hyperref}

\newtheorem{theorem}{Theorem}[section]

\newtheorem{lemma}[theorem]{Lemma}
\newtheorem{corollary}[theorem]{Corollary}
\theoremstyle{remark}
\newtheorem{remark}[theorem]{Remark}

\newcommand{\F}{\mathbb F}
\newcommand{\Leib}{\operatorname{Leib}}
\newcommand{\Aut}{\operatorname{Aut}}

\newcommand{\zleft}{\zeta^{\mathrm{left}}}

\newcommand{\Fx}{\F^{\times}}

\title{Automorphism Groups of Three-Dimensional Leibniz Algebras:\\
A Complete Description}

\author[1]{Leonid A. Kurdachenko}
\author[1]{Oleksandr O. Pypka}
\author[2]{Mykola M. Semko}

\affil[1]{Oles Honchar Dnipro National University, Dnipro, Ukraine}
\affil[2]{State Tax University, Irpin, Ukraine}

\date{}

\begin{document}

\maketitle

\vspace{-3.5em}

\begin{center}
\small
E-mail addresses:\\[2pt]
\href{mailto:lkurdachenko@gmail.com}{lkurdachenko@gmail.com}
(L.A. Kurdachenko)\\
\href{mailto:sasha.pypka@gmail.com}{sasha.pypka@gmail.com}
(O.O. Pypka)\\
\href{mailto:dr.mykola.semko@gmail.com}{dr.mykola.semko@gmail.com}
(M.M. Semko)
\end{center}

\begin{abstract}
Automorphism groups are among the most natural structural invariants of an algebra, and their explicit determination is a basic problem in the structure theory of Leibniz algebras.  In a series of earlier papers, automorphism groups were obtained for several classes of low-dimensional, one-generated, nilpotent, and non-nilpotent Leibniz algebras.  In the present paper we complete this picture for three-dimensional non-Lie left Leibniz algebras over arbitrary fields.  We use a refined organization of the three-dimensional classification into sixteen isomorphism types, including parameterized families.  Previously known cases are not recomputed; instead, they are incorporated by reference, with changes of basis recorded when necessary.  For the remaining types we determine the automorphism groups explicitly through a direct analysis of the automorphism conditions.  Particular attention is paid to exceptional parameter values and to characteristic two.  A final table summarizes the automorphism group of every type in the classification.
\end{abstract}

\medskip
\noindent\textbf{Keywords.} Leibniz algebra; automorphism group; low-dimensional algebra; one-generated Leibniz algebra; arbitrary field.

\medskip
\noindent\textbf{2020 Mathematics Subject Classification.} 17A32, 17A36.

\section{Introduction}

The automorphism group of an algebra records its internal symmetries and is one of the most natural invariants attached to its multiplication.  Besides its intrinsic group-theoretic interest, the automorphism group preserves characteristic ideals and canonical series, acts on subalgebras and ideals, and often provides useful information for classification and isomorphism problems.  In the setting of Leibniz algebras, general properties of automorphisms and derivations, including analogues of the multiplicative and Jordan--Chevalley decompositions in the finite-dimensional complex case, were studied by Ladra, Rikhsiboev and Turdibaev~\cite{Ladra}.  Automorphisms also play a substantial role in the structural theory developed for Leibniz algebras; see, for example,~\cite{Monograph} and the references therein.

The determination of automorphism groups becomes especially concrete in low dimension, where one can combine structural information with explicit matrix calculations.  In our earlier paper~\cite{Old2022}, together with related arbitrary-field classification work such as~\cite{Rakhimov2018}, a detailed analysis of three-dimensional non-Lie Leibniz algebras over arbitrary fields was carried out.  A subsequent reconsideration of the case division in that paper, together with the identification of isomorphic presentations and a refinement of several parameter and characteristic restrictions, allows the resulting list to be organized more economically into sixteen types, including parameterized families.  We use this refined organization throughout the present paper.  The one- and two-dimensional cases are elementary and have already been treated in the low-dimensional automorphism literature~\cite{AutLowDims,AutCyclic}; we therefore restrict attention here to dimension three.

There is by now a substantial collection of results on automorphism groups of Leibniz algebras.  Finite-dimensional one-generated Leibniz algebras were studied in~\cite{AutCyclic2022,AutCyclic}.  Automorphism groups of several low-dimensional and nilpotent Leibniz algebras were obtained in~\cite{AutLowDims,AutLei3,AutAnisotropic,AutNilpotent}, while non-nilpotent three-dimensional cases were considered in~\cite{AutNonNilpotent,AutNonNilpotent2026}.  Di Bartolo, La Rosa and Mancini~\cite{DiBartolo} described, over fields of characteristic different from two, the automorphism group of the non-nilpotent non-Lie Leibniz algebra with one-dimensional derived subalgebra, in arbitrary finite dimension.  More recently, Kaygorodov and Lopatin~\cite{KaygorodovLopatin} determined the automorphism group of every three-dimensional non-Lie Leibniz algebra over the complex field as part of their study of polynomial invariants.

These results leave a natural gap.  For a number of three-dimensional types, the automorphism group was known only over the complex field, only under a characteristic restriction, or only through a more general structural description.  The purpose of the present paper is to close this gap over arbitrary fields.  We do not repeat proofs for cases already completely described in the literature.  Instead, Table~\ref{tab:status} records the full list, the available references, and precisely which cases require new arguments.

The main conclusion can be stated as follows.

\begin{theorem}[Main theorem]\label{thm:main}
Let $L$ be a three-dimensional non-Lie left Leibniz algebra over an arbitrary field $\F$.  Then $\Aut(L)$ is explicitly determined by the formulas proved or cited in the present paper and summarized in Table~\ref{tab:summary}.  In particular, all characteristic-two cases and all parameter families in the refined sixteen-type classification are covered.
\end{theorem}

\section{Preliminaries and the sixteen types}

Throughout the paper, $\F$ is an arbitrary field unless a characteristic restriction is stated explicitly.  All Leibniz algebras are \emph{left} Leibniz algebras, so
\begin{equation}\label{eq:Leibniz}
 [x,[y,z]]=[[x,y],z]+[y,[x,z]]
\end{equation}
for all $x,y,z\in L$.  The Leibniz kernel
\[
 \Leib(L)=\langle [x,x]\mid x\in L\rangle_{\F}
\]
is an ideal contained in the left center $\zleft(L)$.  The derived subalgebra is denoted by $[L,L]$ and the center by $\zeta(L)$.

An automorphism of $L$ is an invertible linear map $f:L\to L$ satisfying
\[
 f([x,y])=[f(x),f(y)]
 \qquad (x,y\in L).
\]
We will use the following standard facts in the same form as in our earlier papers on automorphism groups.

\begin{lemma}\label{lem:characteristic}
Let $L$ be a Leibniz algebra and let $f\in\Aut(L)$.  Then
\[
f(\Leib(L))=\Leib(L),\qquad f([L,L])=[L,L],
\]
and $f$ preserves the left center, the right center, and the center of $L$.
\end{lemma}

Let $G=\Aut(L)$ and let $A$ be a $G$-invariant subspace of $L$.  As usual, we put
\[
 C_G(A)=\{f\in G\mid f(a)=a\text{ for every }a\in A\}
\]
and, if $A$ is an ideal of $L$,
\[
 C_G(L/A)=\{f\in G\mid f(x)\in x+A\text{ for every }x\in L\}.
\]

\begin{lemma}\label{lem:normal}
If $A$ is a $G$-invariant subspace of $L$, then $C_G(A)$ is a normal subgroup of $G$.  If $A$ is an ideal of $L$, then $C_G(L/A)$ is also a normal subgroup of $G$.
\end{lemma}

Fix a basis $a_1,a_2,a_3$ of $L$.  By $\Xi$ we denote the canonical monomorphism of $\Aut(L)$ into $GL_3(\F)$ determined by this basis.  Thus, the $j$-th column of $\Xi(f)$ consists of the coordinates of $f(a_j)$.  All products that are not explicitly displayed are zero.  We write $\F^{+}$ for the additive group of $\F$ and $\F^{\times}$ for its multiplicative group.

The three-dimensional non-Lie algebras used in this paper are listed in Table~\ref{tab:status}.  The table simultaneously records the literature status of their automorphism groups.  The parameter identifications displayed in the second column are part of the classification and prevent repetitions of isomorphic algebras.

\begingroup
\small
\setlength{\tabcolsep}{3pt}
\renewcommand{\arraystretch}{1.18}
\begin{longtable}{@{}P{0.09\textwidth}P{0.415\textwidth}P{0.24\textwidth}P{0.20\textwidth}@{}}
\caption{The refined sixteen-type list and the status of the automorphism problem.}\label{tab:status}\\
\toprule
Type & Nonzero products and restrictions & Previous description of $\Aut(L)$ & Status in the present paper\\
\midrule
\endfirsthead
\toprule
Type & Nonzero products and restrictions & Previous description of $\Aut(L)$ & Status in the present paper\\
\midrule
\endhead
\midrule
\multicolumn{4}{r}{\emph{Continued on next page}}\\
\endfoot
\bottomrule
\endlastfoot
$L_1$ & $[a_1,a_1]=a_3$ & \cite{AutLowDims,KaygorodovLopatin} & previously known\\
$L_2$ & $[a_1,a_1]=[a_1,a_2]=a_3$ & \cite{AutLei3,AutNilpotent,KaygorodovLopatin}; an isomorphic presentation is used in part of the literature & previously known\\
$L_3(\alpha)$ & $[a_1,a_1]=[a_2,a_1]=a_3$, $[a_1,a_2]=\alpha a_3$; $\alpha\ne0,-1$, $\alpha\sim\alpha^{-1}$ & \cite{KaygorodovLopatin}, over $\mathbb C$ & extended to arbitrary $\F$\\
$L_4$ & $[a_1,a_1]=[a_2,a_1]=a_3$, $[a_1,a_2]=-a_3$ & \cite{KaygorodovLopatin}, over $\mathbb C$ & extended to arbitrary $\F$\\
$L_5(\beta)$ & $[a_1,a_1]=a_3$, $[a_2,a_2]=\beta a_3$; $X^2+\beta$ has no root; $\beta$ modulo $(\Fx)^2$ & \cite{AutAnisotropic} & previously known\\
$L_6(\delta)$ & $[a_1,a_1]=[a_1,a_2]=a_3$, $[a_2,a_2]=\delta a_3$; $X^2+X+\delta$ has no root & no complete arbitrary-field description located & obtained here\\
$L_7$ & $[a_1,a_1]=a_3$, $[a_1,a_2]=a_2$, $[a_2,a_1]=-a_2$ & \cite{KaygorodovLopatin}, over $\mathbb C$ & extended to arbitrary $\F$\\
$L_8(\lambda)$ & $\operatorname{char}\F=2$; $[a_1,a_1]=\lambda a_3$, $[a_1,a_2]=[a_2,a_1]=a_2$, $[a_2,a_2]=a_3$; $\lambda'=(\lambda+t^2)/s^2$ for $t\in\F$, $s\in\Fx$ & no complete description located & obtained here\\
$L_9$ & $[a_1,a_1]=[a_1,a_3]=a_3$ & \cite{DiBartolo} if $\operatorname{char}\F\ne2$; \cite{KaygorodovLopatin} over $\mathbb C$ & characteristic $2$ completed here\\
$L_{10}(r)$ & $r\in\Fx$; $[a_1,a_1]=a_3$, $[a_1,a_2]=a_2$, $[a_2,a_1]=-a_2$, $[a_1,a_3]=r a_3$ & \cite{KaygorodovLopatin}, over $\mathbb C$ & extended to arbitrary $\F$ and all $r$\\
$L_{11}$ & $\operatorname{char}\F\ne2$; $[a_2,a_2]=a_3$, $[a_1,a_2]=a_2$, $[a_2,a_1]=-a_2$, $[a_1,a_3]=2a_3$ & \cite{KaygorodovLopatin}, over $\mathbb C$ & obtained for arbitrary $\F$, $\operatorname{char}\F\ne2$\\
$L_{12}$ & $[a_1,a_1]=a_2$, $[a_1,a_2]=a_3$ & \cite{AutCyclic2022,AutCyclic,AutLowDims,KaygorodovLopatin} & previously known\\
$L_{13}$ & $[a_1,a_1]=a_2$, $[a_1,a_2]=a_2+a_3$ & \cite{AutCyclic2022,AutCyclic,AutNonNilpotent,KaygorodovLopatin}; different presentations occur & previously known\\
$L_{14}(\rho)$ & $\rho\in\Fx$; $[a_1,a_1]=[a_1,a_2]=a_2$, $[a_1,a_3]=\rho a_3$; $\rho\sim\rho^{-1}$ & \cite{AutNonNilpotent2026} for $\rho=1$; \cite{KaygorodovLopatin} over $\mathbb C$ & general parameter completed here\\
$L_{15}$ & $[a_1,a_1]=a_2$, $[a_1,a_2]=a_2+a_3$, $[a_1,a_3]=a_3$ & \cite{AutCyclic2022,AutCyclic,AutNonNilpotent2026,KaygorodovLopatin}; different presentation in \cite{AutNonNilpotent2026} & previously known\\
$L_{16}(\beta,\gamma)$ & $[a_1,a_1]=a_2$, $[a_1,a_2]=a_3$, $[a_1,a_3]=\beta a_2+\gamma a_3$; $X^2-\gamma X-\beta$ irreducible; $(\beta,\gamma)\sim(c^2\beta,c\gamma)$ & one-generated type-II structure in \cite{AutCyclic2022,AutCyclic} & explicit $3\times3$ form obtained here\\
\end{longtable}
\endgroup

\begin{remark}[Changes of basis in previously published cases]\label{rem:basis}
Some papers cited in Table~\ref{tab:status} use multiplication tables different from the representatives fixed here.  We record a few changes of basis that will be useful when comparing formulas.

If a published representative of $L_2$ is written as
\[
 [e_1,e_1]=[e_2,e_1]=e_3,
\]
then
\[
 a_1=e_1,\qquad a_2=e_1-e_2,\qquad a_3=e_3
\]
gives the table used here.

A common presentation leading to $L_{13}$ is
\[
 [e_1,e_1]=e_3,\qquad [e_1,e_2]=e_2+\lambda e_3.
\]
For any $\lambda\in\F$, the basis
\[
 a_1=e_1+e_2,\qquad
 a_2=e_2+(\lambda+1)e_3,\qquad
 a_3=-e_3
\]
gives $[a_1,a_1]=a_2$ and $[a_1,a_2]=a_2+a_3$.

The algebra $L_{14}(1)$ is written in~\cite{AutNonNilpotent2026} in the equivalent form
\[
 [e_1,e_1]=[e_1,e_3]=e_3,\qquad [e_1,e_2]=e_2;
\]
one only interchanges $e_2$ and $e_3$.

Finally, a presentation used for $L_{15}$ is
\[
 [e_1,e_1]=[e_1,e_3]=e_3,\qquad [e_1,e_2]=e_2+\lambda e_3,
 \qquad \lambda\ne0.
\]
The change
\[
 a_1=e_1+e_2-(1+\lambda)e_3,\qquad
 a_2=e_2,\qquad
 a_3=\lambda e_3
\]
gives the table of $L_{15}$ in Table~\ref{tab:status}.
\end{remark}

\section{Automorphism groups in the case \texorpdfstring{$\dim_{\F}\Leib(L)=1$}{dim Leib(L)=1} and \texorpdfstring{$L$}{L} is nilpotent}

The automorphism groups of $L_1$, $L_2$, and $L_5(\beta)$ were described previously; see Table~\ref{tab:status}.  Therefore, in this section we consider only $L_3(\alpha)$, $L_4$, and $L_6(\delta)$.

\subsection{The Leibniz algebra \texorpdfstring{$L_3(\alpha)$}{L3(alpha)}}

\begin{theorem}\label{thm:L3}
Let $L=L_3(\alpha)$, where $\alpha\ne0,-1$, and let $G=\Aut(L)$.

If $\alpha\ne1$, then $\Xi(G)$ consists of all nondegenerate matrices of the form
\begin{equation}\label{eq:AutL3generic}
\begin{pmatrix}
 \alpha_1&0&0\\[1mm]
 \dfrac{\beta_2-\alpha_1}{1+\alpha}&\beta_2&0\\[2mm]
 \alpha_3&\beta_3&\alpha_1\beta_2
\end{pmatrix},
\qquad
\alpha_1,\beta_2\in\F^{\times},\quad \alpha_3,\beta_3\in\F.
\end{equation}
Moreover, $G$ contains a normal subgroup isomorphic to $\F^{+}\times\F^{+}$ and a subgroup isomorphic to $\F^{\times}\times\F^{\times}$, and $G$ is their semidirect product.

If $\alpha=1$, then $\operatorname{char}\F\ne2$ and, in addition to the matrices in~\eqref{eq:AutL3generic}, the group $\Xi(G)$ contains all nondegenerate matrices of the form
\begin{equation}\label{eq:AutL3exceptional}
\begin{pmatrix}
 \alpha_1&2\alpha_1&0\\
 \alpha_2&-\alpha_1&0\\
 \alpha_3&\beta_3&\alpha_1(\alpha_1+2\alpha_2)
\end{pmatrix},
\qquad
\alpha_1(\alpha_1+2\alpha_2)\ne0.
\end{equation}
In this case $G$ is a semidirect product of a normal subgroup isomorphic to $\F^{+}\times\F^{+}$ and a subgroup which is a semidirect product of $\F^{\times}\times\F^{\times}$ by a group of order two.  The nontrivial element of the latter group interchanges the two multiplicative factors.
\end{theorem}

\begin{proof}
Let $f\in G$.  By Lemma~\ref{lem:characteristic},
\[
 f(\Leib(L))=\Leib(L)=\F a_3.
\]
Hence
\[
\begin{aligned}
 f(a_1)&=\alpha_1a_1+\alpha_2a_2+\alpha_3a_3,\\
 f(a_2)&=\beta_1a_1+\beta_2a_2+\beta_3a_3,\\
 f(a_3)&=\gamma_3a_3.
\end{aligned}
\]
Since $f$ is an automorphism,
\[
 \gamma_3\ne0,
 \qquad
 \alpha_1\beta_2-\alpha_2\beta_1\ne0.
\]
We now use the defining products of $L$ successively.

First,
\[
\begin{aligned}
 \gamma_3a_3
 &=f(a_3)=f([a_1,a_1])=[f(a_1),f(a_1)]\\
 &=[\alpha_1a_1+\alpha_2a_2+\alpha_3a_3,
    \alpha_1a_1+\alpha_2a_2+\alpha_3a_3]\\
 &=\alpha_1^2[a_1,a_1]
   +\alpha_1\alpha_2[a_1,a_2]
   +\alpha_2\alpha_1[a_2,a_1]\\
 &=\bigl(\alpha_1^2+(1+\alpha)\alpha_1\alpha_2\bigr)a_3.
\end{aligned}
\]
Therefore
\begin{equation}\label{eq:L3-1}
 \gamma_3=\alpha_1^2+(1+\alpha)\alpha_1\alpha_2.
\end{equation}
Next,
\[
\begin{aligned}
 \gamma_3a_3
 &=f([a_2,a_1])=[f(a_2),f(a_1)]\\
 &=[\beta_1a_1+\beta_2a_2+\beta_3a_3,
    \alpha_1a_1+\alpha_2a_2+\alpha_3a_3]\\
 &=\beta_1\alpha_1[a_1,a_1]
   +\beta_1\alpha_2[a_1,a_2]
   +\beta_2\alpha_1[a_2,a_1]\\
 &=\bigl(\beta_1\alpha_1+\alpha\beta_1\alpha_2
          +\beta_2\alpha_1\bigr)a_3,
\end{aligned}
\]
so
\begin{equation}\label{eq:L3-2}
 \gamma_3=\beta_1\alpha_1+\beta_2\alpha_1+\alpha\beta_1\alpha_2.
\end{equation}
Further,
\[
\begin{aligned}
 \alpha\gamma_3a_3
 &=f([a_1,a_2])=[f(a_1),f(a_2)]\\
 &=\alpha_1\beta_1[a_1,a_1]
   +\alpha_1\beta_2[a_1,a_2]
   +\alpha_2\beta_1[a_2,a_1]\\
 &=\bigl(\alpha_1\beta_1+\alpha\alpha_1\beta_2
          +\alpha_2\beta_1\bigr)a_3,
\end{aligned}
\]
and hence
\begin{equation}\label{eq:L3-3}
 \alpha\gamma_3=\alpha_1\beta_1+\alpha\alpha_1\beta_2+\alpha_2\beta_1.
\end{equation}
Finally,
\[
\begin{aligned}
 0=f([a_2,a_2])=[f(a_2),f(a_2)]
 &=\bigl(\beta_1^2+(1+\alpha)\beta_1\beta_2\bigr)a_3.
\end{aligned}
\]
Thus
\begin{equation}\label{eq:L3-4}
 \beta_1\bigl(\beta_1+(1+\alpha)\beta_2\bigr)=0.
\end{equation}

Suppose first that $\beta_1=0$.  The nondegeneracy of the matrix of $f$ implies that
$\alpha_1\ne0$ and $\beta_2\ne0$.  Equality~\eqref{eq:L3-2} gives
\[
 \gamma_3=\alpha_1\beta_2.
\]
Combining this equality with~\eqref{eq:L3-1}, we obtain
\[
 \alpha_1\beta_2
 =\alpha_1^2+(1+\alpha)\alpha_1\alpha_2.
\]
Since $\alpha_1\ne0$ and $1+\alpha\ne0$, it follows that
\[
 \alpha_2=\frac{\beta_2-\alpha_1}{1+\alpha}.
\]
Substitution in~\eqref{eq:L3-3} gives an identity.  Therefore $\Xi(f)$ has the form~\eqref{eq:AutL3generic}.

Suppose now that $\beta_1\ne0$.  By~\eqref{eq:L3-4},
\begin{equation}\label{eq:L3-beta1}
 \beta_1=-(1+\alpha)\beta_2.
\end{equation}
In particular, $\beta_2\ne0$.  Moreover,
\[
 \alpha_1\beta_2-\alpha_2\beta_1
 =\beta_2\bigl(\alpha_1+(1+\alpha)\alpha_2\bigr)\ne0.
\]
Using~\eqref{eq:L3-beta1} in~\eqref{eq:L3-1} and~\eqref{eq:L3-2}, we get
\[
 \gamma_3=\alpha_1\bigl(\alpha_1+(1+\alpha)\alpha_2\bigr)
          =-\alpha\beta_2\bigl(\alpha_1+(1+\alpha)\alpha_2\bigr).
\]
The last factor is nonzero, and hence
\[
 \alpha_1=-\alpha\beta_2.
\]
Substituting this equality and~\eqref{eq:L3-beta1} into~\eqref{eq:L3-3}, we obtain
\[
 \beta_2(\alpha-1)(\alpha+1)
 \bigl(\alpha_1+(1+\alpha)\alpha_2\bigr)=0.
\]
Since $\beta_2\ne0$, $\alpha\ne-1$, and the last factor is nonzero, we have $\alpha=1$.
Consequently $\operatorname{char}\F\ne2$.  Moreover,
\[
 \alpha_1=-\beta_2,
 \qquad
 \beta_1=2\alpha_1,
\]
and the matrix of $f$ has the form~\eqref{eq:AutL3exceptional}.  Its determinant is
\[
 -\alpha_1^2(\alpha_1+2\alpha_2)^2,
\]
so the condition in~\eqref{eq:AutL3exceptional} is exactly the nondegeneracy condition.

We verify the converse.  By bilinearity it is enough to check the products of the basis elements.  For a matrix of the form~\eqref{eq:AutL3generic}, put
\[
 q=\frac{\beta_2-\alpha_1}{1+\alpha}.
\]
Then
\[
\begin{aligned}
 [f(a_1),f(a_1)]
 &=\bigl(\alpha_1^2+(1+\alpha)\alpha_1q\bigr)a_3
 =\alpha_1\beta_2a_3=f(a_3),\\
 [f(a_2),f(a_1)]
 &=\alpha_1\beta_2a_3=f(a_3),\\
 [f(a_1),f(a_2)]
 &=\alpha\alpha_1\beta_2a_3=\alpha f(a_3),\\
 [f(a_2),f(a_2)]&=0.
\end{aligned}
\]
All products containing $f(a_3)$ are zero because $a_3$ belongs to the left and right center of this algebra.  Hence every nondegenerate matrix in~\eqref{eq:AutL3generic} defines an automorphism.

Now let $\alpha=1$ and let $f$ have a matrix of the form~\eqref{eq:AutL3exceptional}.  Thus
\[
\begin{aligned}
 f(a_1)&=\alpha_1a_1+\alpha_2a_2+\alpha_3a_3,\\
 f(a_2)&=2\alpha_1a_1-\alpha_1a_2+\beta_3a_3,\\
 f(a_3)&=\alpha_1(\alpha_1+2\alpha_2)a_3.
\end{aligned}
\]
Using
\[
 [a_1,a_1]=[a_2,a_1]=[a_1,a_2]=a_3,
 \qquad [a_2,a_2]=0,
\]
we obtain
\[
\begin{aligned}
 [f(a_1),f(a_1)]
 &=\bigl(\alpha_1^2+2\alpha_1\alpha_2\bigr)a_3
 =f(a_3),\\
 [f(a_2),f(a_1)]
 &=\bigl(2\alpha_1^2+2\alpha_1\alpha_2-\alpha_1^2\bigr)a_3
 =f(a_3),\\
 [f(a_1),f(a_2)]
 &=\bigl(2\alpha_1^2-\alpha_1^2+2\alpha_1\alpha_2\bigr)a_3
 =f(a_3),\\
 [f(a_2),f(a_2)]
 &=\bigl(4\alpha_1^2-2\alpha_1^2-2\alpha_1^2\bigr)a_3=0.
\end{aligned}
\]
Again all products containing $f(a_3)$ are zero.  Therefore every nondegenerate matrix in~\eqref{eq:AutL3exceptional} also defines an automorphism.

It remains to describe the group structure.  Put
\[
 C=C_G(L/\F a_3).
\]
Then $C$ consists of all automorphisms whose matrices are
\[
 \begin{pmatrix}
 1&0&0\\0&1&0\\ r&u&1
 \end{pmatrix},
 \qquad r,u\in\F.
\]
By Lemma~\ref{lem:normal}, $C\trianglelefteq G$.  Moreover,
\[
 \begin{pmatrix}1&0&0\\0&1&0\\r&u&1\end{pmatrix}
 \begin{pmatrix}1&0&0\\0&1&0\\s&v&1\end{pmatrix}
 =
 \begin{pmatrix}1&0&0\\0&1&0\\r+s&u+v&1\end{pmatrix}.
\]
Hence the correspondence of this matrix to $(r,u)$ is an isomorphism
\[
 C\cong\F^{+}\times\F^{+}.
\]

Assume first that $\alpha\ne1$.  Let $D$ be the set of all automorphisms whose matrices have the form
\[
 \begin{pmatrix}
 \lambda&0&0\\[1mm]
 \dfrac{\mu-\lambda}{1+\alpha}&\mu&0\\[2mm]
 0&0&\lambda\mu
 \end{pmatrix},
 \qquad \lambda,\mu\in\F^{\times}.
\]
A direct multiplication gives
\[
\begin{pmatrix}
 \lambda&0&0\\ \dfrac{\mu-\lambda}{1+\alpha}&\mu&0\\0&0&\lambda\mu
\end{pmatrix}
\begin{pmatrix}
 \lambda'&0&0\\ \dfrac{\mu'-\lambda'}{1+\alpha}&\mu'&0\\0&0&\lambda'\mu'
\end{pmatrix}
=
\begin{pmatrix}
 \lambda\lambda'&0&0\\
 \dfrac{\mu\mu'-\lambda\lambda'}{1+\alpha}&\mu\mu'&0\\
 0&0&\lambda\lambda'\mu\mu'
\end{pmatrix}.
\]
Thus $D$ is a subgroup and the correspondence of this matrix to $(\lambda,\mu)$ is an isomorphism
$D\cong\F^{\times}\times\F^{\times}$.  Every matrix in~\eqref{eq:AutL3generic} is a product of an element of $C$ and an element of $D$, and $C\cap D=\{1\}$.  Therefore
\[
 G=C\rtimes D.
\]
For completeness, direct conjugation gives
\[
\begin{pmatrix}
 \lambda&0&0\\ \dfrac{\mu-\lambda}{1+\alpha}&\mu&0\\0&0&\lambda\mu
\end{pmatrix}
\begin{pmatrix}1&0&0\\0&1&0\\r&u&1\end{pmatrix}
\begin{pmatrix}
 \lambda&0&0\\ \dfrac{\mu-\lambda}{1+\alpha}&\mu&0\\0&0&\lambda\mu
\end{pmatrix}^{-1}
=
\begin{pmatrix}
1&0&0\\0&1&0\\
\mu r-\dfrac{\mu-\lambda}{1+\alpha}u&\lambda u&1
\end{pmatrix}.
\]

Finally let $\alpha=1$.  We use the same subgroup $D$ and consider
\[
 \tau=\begin{pmatrix}1&2&0\\0&-1&0\\0&0&1\end{pmatrix}.
\]
Then $\tau^2=1$.  Direct multiplication shows that
\[
 \tau
 \begin{pmatrix}
 \lambda&0&0\\ \dfrac{\mu-\lambda}{2}&\mu&0\\0&0&\lambda\mu
 \end{pmatrix}
 \tau^{-1}
 =
 \begin{pmatrix}
 \mu&0&0\\ \dfrac{\lambda-\mu}{2}&\lambda&0\\0&0&\lambda\mu
 \end{pmatrix}.
\]
Thus conjugation by $\tau$ interchanges the two multiplicative parameters.  Put
\[
 H=\langle D,\tau\rangle.
\]
Since $\tau$ normalizes $D$, we have $D\trianglelefteq H$, $D\cap\langle\tau\rangle=\{1\}$ and
\[
 H=D\langle\tau\rangle.
\]
Hence $H$ is a semidirect product of $D\cong\F^{\times}\times\F^{\times}$ by the group $\langle\tau\rangle$ of order two.

It remains only to check that this subgroup is a complement to $C$.  Let an automorphism from the second family~\eqref{eq:AutL3exceptional} be given.  Put
\[
 \lambda=\alpha_1,
 \qquad
 \mu=\alpha_1+2\alpha_2.
\]
Then $\lambda,\mu\ne0$, and
\[
 \begin{pmatrix}
 \lambda&0&0\\[1mm]
 \dfrac{\mu-\lambda}{2}&\mu&0\\[2mm]
 0&0&\lambda\mu
 \end{pmatrix}\tau
 =
 \begin{pmatrix}
 \alpha_1&2\alpha_1&0\\
 \alpha_2&-\alpha_1&0\\
 0&0&\alpha_1(\alpha_1+2\alpha_2)
 \end{pmatrix}.
\]
Since $\alpha_1(\alpha_1+2\alpha_2)\ne0$, there exist unique $r,u\in\F$ such that
\[
 \alpha_1r+\alpha_2u=\alpha_3,
 \qquad
 2\alpha_1r-\alpha_1u=\beta_3.
\]
Therefore every matrix in~\eqref{eq:AutL3exceptional} belongs to $CH$.  The matrices in~\eqref{eq:AutL3generic} already belong to $CD\le CH$, so $G=CH$.  From the displayed forms of $C$ and $H$ we also have $C\cap H=\{1\}$.  Since $C\trianglelefteq G$,
\[
 G=C\rtimes H,
\]
where $H$ is the semidirect product described above.  This completes the proof.
\end{proof}

\subsection{The Leibniz algebra \texorpdfstring{$L_4$}{L4}}

\begin{theorem}\label{thm:L4}
Let $L=L_4$ and $G=\Aut(L)$.  Then $\Xi(G)$ consists of all matrices
\begin{equation}\label{eq:AutL4}
\begin{pmatrix}
 \alpha_1&0&0\\
 \alpha_2&\alpha_1&0\\
 \alpha_3&\beta_3&\alpha_1^2
\end{pmatrix},
\qquad
\alpha_1\in\F^{\times},\quad \alpha_2,\alpha_3,\beta_3\in\F.
\end{equation}
This formula is valid in every characteristic.  Moreover, $G$ is a semidirect product of a normal subgroup isomorphic to $\F^{+}\times\F^{+}$ and a subgroup isomorphic to $\F^{+}\times\F^{\times}$.
\end{theorem}

\begin{proof}
Let $f\in G$.  Since $\Leib(L)=[L,L]=\F a_3$ is invariant under every automorphism, we may write
\[
\begin{aligned}
 f(a_1)&=\alpha_1a_1+\alpha_2a_2+\alpha_3a_3,\\
 f(a_2)&=\beta_1a_1+\beta_2a_2+\beta_3a_3,\\
 f(a_3)&=\gamma_3a_3.
\end{aligned}
\]
The matrix of $f$ is nondegenerate, so
\[
 \gamma_3\ne0,
 \qquad
 \alpha_1\beta_2-\alpha_2\beta_1\ne0.
\]
The equality $[a_2,a_2]=0$ gives
\[
\begin{aligned}
0&=[f(a_2),f(a_2)]\\
 &=\beta_1^2[a_1,a_1]
   +\beta_1\beta_2[a_1,a_2]
   +\beta_2\beta_1[a_2,a_1]\\
 &=\beta_1^2a_3.
\end{aligned}
\]
Hence $\beta_1=0$.  Consequently $\alpha_1\ne0$ and $\beta_2\ne0$.
Next,
\[
\begin{aligned}
 \gamma_3a_3
 &=f([a_1,a_1])=[f(a_1),f(a_1)]\\
 &=\alpha_1^2a_3
   -\alpha_1\alpha_2a_3
   +\alpha_2\alpha_1a_3
 =\alpha_1^2a_3,
\end{aligned}
\]
so $\gamma_3=\alpha_1^2$.  Also
\[
\begin{aligned}
 \gamma_3a_3
 &=f([a_2,a_1])=[f(a_2),f(a_1)]
 =\beta_2\alpha_1a_3.
\end{aligned}
\]
Thus $\beta_2=\alpha_1$.  Finally,
\[
 [f(a_1),f(a_2)]=-\alpha_1^2a_3=-f(a_3),
\]
so the relation $[a_1,a_2]=-a_3$ is automatically satisfied.  No condition is imposed on $\alpha_2,\alpha_3,\beta_3$.  We have therefore obtained exactly the matrices in~\eqref{eq:AutL4}.

Conversely, let $f$ be a nondegenerate linear transformation having the matrix~\eqref{eq:AutL4}.  Direct calculation gives
\[
\begin{aligned}
 [f(a_1),f(a_1)]&=\alpha_1^2a_3=f(a_3),\\
 [f(a_2),f(a_1)]&=\alpha_1^2a_3=f(a_3),\\
 [f(a_1),f(a_2)]&=-\alpha_1^2a_3=-f(a_3),\\
 [f(a_2),f(a_2)]&=0.
\end{aligned}
\]
Every product involving $f(a_3)$ is zero.  Hence $f$ preserves all products of the basis elements and therefore is an automorphism.  Notice that no division by $2$ occurs, so the same argument is valid in characteristic two.

We now describe the group structure.  Let $C=C_G(L/\F a_3)$.  Then
\[
 \Xi(C)=\left\{
 \begin{pmatrix}1&0&0\\0&1&0\\r&u&1\end{pmatrix}:r,u\in\F
 \right\},
\]
and $C\trianglelefteq G$ by Lemma~\ref{lem:normal}.  Clearly $C\cong\F^{+}\times\F^{+}$.
Let $D$ be the subgroup formed by the matrices
\[
 \begin{pmatrix}
 \lambda&0&0\\c\lambda&\lambda&0\\0&0&\lambda^2
 \end{pmatrix},
 \qquad c\in\F,\quad \lambda\in\F^{\times}.
\]
Multiplying two such matrices, we obtain
\[
\begin{pmatrix}\lambda&0&0\\c\lambda&\lambda&0\\0&0&\lambda^2\end{pmatrix}
\begin{pmatrix}\mu&0&0\\d\mu&\mu&0\\0&0&\mu^2\end{pmatrix}
=
\begin{pmatrix}\lambda\mu&0&0\\(c+d)\lambda\mu&\lambda\mu&0\\0&0&(\lambda\mu)^2\end{pmatrix}.
\]
Thus $D\cong\F^{+}\times\F^{\times}$.  Every element of $G$ can be written as a product of an element of $C$ and an element of $D$, and $C\cap D=\{1\}$.  Consequently
\[
 G=C\rtimes D.
\]
The action is obtained directly from
\[
\begin{pmatrix}\lambda&0&0\\c\lambda&\lambda&0\\0&0&\lambda^2\end{pmatrix}
\begin{pmatrix}1&0&0\\0&1&0\\r&u&1\end{pmatrix}
\begin{pmatrix}\lambda&0&0\\c\lambda&\lambda&0\\0&0&\lambda^2\end{pmatrix}^{-1}
=
\begin{pmatrix}1&0&0\\0&1&0\\\lambda(r-cu)&\lambda u&1\end{pmatrix}.
\]
This proves the theorem.
\end{proof}

\subsection{The Leibniz algebra \texorpdfstring{$L_6(\delta)$}{L6(delta)}}

\begin{theorem}\label{thm:L6}
Let $L=L_6(\delta)$, where the polynomial $X^2+X+\delta$ has no root in $\F$, and let $G=\Aut(L)$.  Then $\Xi(G)$ consists of all matrices
\begin{equation}\label{eq:AutL6}
\begin{pmatrix}
 \alpha_1&-\delta\alpha_2&0\\
 \alpha_2&\alpha_1+\alpha_2&0\\
 \alpha_3&\beta_3&\alpha_1^2+\alpha_1\alpha_2+\delta\alpha_2^2
\end{pmatrix},
\end{equation}
where $\alpha_1,\alpha_2,\alpha_3,\beta_3\in\F$ and $(\alpha_1,\alpha_2)\ne(0,0)$.

Let $E=\F(\theta)$, where $\theta^2-\theta+\delta=0$.  Then $G$ is a semidirect product of a normal subgroup isomorphic to $\F^{+}\times\F^{+}$ and a subgroup isomorphic to the multiplicative group $E^{\times}$.
\end{theorem}

\begin{proof}
Let $f\in G$.  As $\Leib(L)=[L,L]=\F a_3$ is invariant under $f$, we have
\[
\begin{aligned}
 f(a_1)&=\alpha_1a_1+\alpha_2a_2+\alpha_3a_3,\\
 f(a_2)&=\beta_1a_1+\beta_2a_2+\beta_3a_3,\\
 f(a_3)&=\gamma_3a_3.
\end{aligned}
\]
The nondegeneracy of $f$ gives
\[
 \gamma_3\ne0,
 \qquad
 \alpha_1\beta_2-\alpha_2\beta_1\ne0.
\]
From $[a_1,a_1]=a_3$ we obtain
\[
\begin{aligned}
 \gamma_3a_3
 &= [f(a_1),f(a_1)]\\
 &=\alpha_1^2[a_1,a_1]
   +\alpha_1\alpha_2[a_1,a_2]
   +\alpha_2^2[a_2,a_2]\\
 &=\bigl(\alpha_1^2+\alpha_1\alpha_2+\delta\alpha_2^2\bigr)a_3.
\end{aligned}
\]
Thus
\begin{equation}\label{eq:L6-1}
 \gamma_3=\alpha_1^2+\alpha_1\alpha_2+\delta\alpha_2^2.
\end{equation}
Using $[a_2,a_1]=0$, we get
\[
\begin{aligned}
0&=[f(a_2),f(a_1)]\\
 &=\beta_1\alpha_1[a_1,a_1]
   +\beta_1\alpha_2[a_1,a_2]
   +\beta_2\alpha_2[a_2,a_2]\\
 &=\bigl(\beta_1(\alpha_1+\alpha_2)+\delta\alpha_2\beta_2\bigr)a_3,
\end{aligned}
\]
whence
\begin{equation}\label{eq:L6-2}
 \beta_1(\alpha_1+\alpha_2)+\delta\alpha_2\beta_2=0.
\end{equation}
Further,
\[
\begin{aligned}
 \gamma_3a_3
 &=f([a_1,a_2])=[f(a_1),f(a_2)]\\
 &=\bigl(\alpha_1\beta_1+\alpha_1\beta_2+\delta\alpha_2\beta_2\bigr)a_3,
\end{aligned}
\]
so
\begin{equation}\label{eq:L6-3}
 \alpha_1\beta_1+(\alpha_1+\delta\alpha_2)\beta_2=\gamma_3.
\end{equation}
Equations~\eqref{eq:L6-2} and~\eqref{eq:L6-3} form a linear system in $\beta_1,\beta_2$.  Its determinant is
\[
 (\alpha_1+\alpha_2)(\alpha_1+\delta\alpha_2)-\delta\alpha_1\alpha_2
 =\alpha_1^2+\alpha_1\alpha_2+\delta\alpha_2^2
 =\gamma_3.
\]
Since $\gamma_3\ne0$, the system has a unique solution.  Applying Cramer's rule and using that its determinant is $\gamma_3$, we obtain
\[
 \beta_1=
 \frac{
 \begin{vmatrix}
 0&\delta\alpha_2\\
 \gamma_3&\alpha_1+\delta\alpha_2
 \end{vmatrix}}
 {\gamma_3}
 =-\delta\alpha_2
\]
and
\[
 \beta_2=
 \frac{
 \begin{vmatrix}
 \alpha_1+\alpha_2&0\\
 \alpha_1&\gamma_3
 \end{vmatrix}}
 {\gamma_3}
 =\alpha_1+\alpha_2.
\]
Finally,
\[
\begin{aligned}
 [f(a_2),f(a_2)]
 &=\bigl(\beta_1^2+\beta_1\beta_2+\delta\beta_2^2\bigr)a_3\\
 &=\delta\bigl(\alpha_1^2+\alpha_1\alpha_2+\delta\alpha_2^2\bigr)a_3
 =\delta f(a_3),
\end{aligned}
\]
so the last defining relation gives no additional restriction.  We have obtained the matrices in~\eqref{eq:AutL6}.

Conversely, suppose a linear transformation has the matrix~\eqref{eq:AutL6}.  If $(\alpha_1,\alpha_2)\ne(0,0)$, then
\[
 \alpha_1^2+\alpha_1\alpha_2+\delta\alpha_2^2\ne0.
\]
Indeed, if $\alpha_2\ne0$, then $\alpha_1/\alpha_2$ would otherwise be a root of $X^2+X+\delta$; and if $\alpha_2=0$, then $\alpha_1\ne0$.  Hence the displayed matrix is nondegenerate.  Direct calculation gives
\[
\begin{aligned}
 [f(a_1),f(a_1)]&=f(a_3),\\
 [f(a_1),f(a_2)]&=f(a_3),\\
 [f(a_2),f(a_1)]&=0,\\
 [f(a_2),f(a_2)]&=\delta f(a_3).
\end{aligned}
\]
All products containing $f(a_3)$ are zero.  Thus the transformation is an automorphism.

We now describe the group structure.  Let
\[
 C=C_G(L/\F a_3).
\]
Then $C$ is a normal subgroup of $G$ and
\[
 \Xi(C)=\left\{
 \begin{pmatrix}1&0&0\\0&1&0\\r&u&1\end{pmatrix}:r,u\in\F
 \right\}.
\]
Therefore $C\cong\F^{+}\times\F^{+}$.
Let $D$ be the subgroup consisting of the matrices
\[
 \begin{pmatrix}
 p&-\delta q&0\\
 q&p+q&0\\
 0&0&p^2+pq+\delta q^2
 \end{pmatrix},
 \qquad (p,q)\ne(0,0).
\]
Every element of $G$ is a product of an element of $C$ and an element of $D$, and $C\cap D=\{1\}$.  Hence $G=C\rtimes D$.

It remains to identify $D$.  Since $X^2-X+\delta$ is irreducible whenever $X^2+X+\delta$ is irreducible, let $E=\F(\theta)$, where
\[
 \theta^2-\theta+\delta=0.
\]
Define
\[
 \Phi:E^{\times}\longrightarrow D
\]
by
\[
 \Phi(p+q\theta)=
 \begin{pmatrix}
 p&-\delta q&0\\
 q&p+q&0\\
 0&0&p^2+pq+\delta q^2
 \end{pmatrix}.
\]
If $p,q,p',q'\in\F$, then
\[
 (p+q\theta)(p'+q'\theta)
 =(pp'-\delta qq')+(pq'+qp'+qq')\theta.
\]
On the other hand, direct multiplication of the corresponding matrices gives the matrix obtained from the same two coefficients
\[
 pp'-\delta qq',
 \qquad
 pq'+qp'+qq'.
\]
Thus $\Phi$ is a homomorphism.  It is injective, because $\Phi(p+q\theta)=I$ implies $p=1$ and $q=0$, and it is surjective by the definition of $D$.  Therefore
\[
 D\cong E^{\times}.
\]
Finally, let an element of $C$ have parameters $r,u$, and let an element of $D$ have parameters $p,q$.  Direct conjugation shows that the resulting element of $C$ has the matrix
\[
\begin{pmatrix}
1&0&0\\
0&1&0\\
r(p+q)-uq&r\delta q+up&1
\end{pmatrix}.
\]
Thus the action of $D$ on $C$ is also explicitly determined, and the proof is complete.
\end{proof}

\section{Automorphism groups of the remaining algebras with one-dimensional Leibniz kernel}

\subsection{The Leibniz algebra \texorpdfstring{$L_7$}{L7}}

\begin{theorem}\label{thm:L7}
Let $L=L_7$ and $G=\Aut(L)$.  Then $\Xi(G)$ consists of all matrices
\begin{equation}\label{eq:AutL7}
\begin{pmatrix}
1&0&0\\
\alpha_2&\beta_2&0\\
\alpha_3&0&1
\end{pmatrix},
\qquad \beta_2\in\F^{\times},\quad \alpha_2,\alpha_3\in\F.
\end{equation}
Moreover, $G$ is a direct product of a subgroup isomorphic to $\F^{+}$ and a subgroup which is a semidirect product of $\F^{+}$ by $\F^{\times}$; the multiplicative group acts on the additive group by ordinary scalar multiplication.
\end{theorem}

\begin{proof}
Let $f\in G$.  We have
\[
 \Leib(L)=\F a_3,
 \qquad
 [L,L]=\F a_2\oplus\F a_3.
\]
Both subspaces are invariant under every automorphism.  Hence
\[
\begin{aligned}
 f(a_1)&=\alpha_1a_1+\alpha_2a_2+\alpha_3a_3,\\
 f(a_2)&=\beta_2a_2+\beta_3a_3,\\
 f(a_3)&=\gamma_3a_3,
\end{aligned}
\]
where $\alpha_1\beta_2\gamma_3\ne0$.
Using $[a_2,a_1]=-a_2$, we obtain
\[
\begin{aligned}
 -\beta_2a_2-\beta_3a_3
 &=f([a_2,a_1])=[f(a_2),f(a_1)]\\
 &=[\beta_2a_2+\beta_3a_3,
    \alpha_1a_1+\alpha_2a_2+\alpha_3a_3]\\
 &=-\alpha_1\beta_2a_2.
\end{aligned}
\]
Since $\beta_2\ne0$, it follows that
\[
 \alpha_1=1,
 \qquad
 \beta_3=0.
\]
Next,
\[
\begin{aligned}
 \gamma_3a_3
 &=f([a_1,a_1])=[f(a_1),f(a_1)]\\
 &=[a_1+\alpha_2a_2+\alpha_3a_3,
    a_1+\alpha_2a_2+\alpha_3a_3]\\
 &=a_3+\alpha_2a_2-\alpha_2a_2=a_3.
\end{aligned}
\]
Thus $\gamma_3=1$.  The relation $[a_1,a_2]=a_2$ is then preserved automatically:
\[
 [f(a_1),f(a_2)]=[a_1+\alpha_2a_2+\alpha_3a_3,\beta_2a_2]
 =\beta_2a_2=f(a_2).
\]
Therefore the matrix of $f$ is exactly~\eqref{eq:AutL7}.

Conversely, let $f$ have the matrix~\eqref{eq:AutL7}.  Then
\[
\begin{aligned}
 [f(a_1),f(a_1)]&=a_3=f(a_3),\\
 [f(a_1),f(a_2)]&=\beta_2a_2=f(a_2),\\
 [f(a_2),f(a_1)]&=-\beta_2a_2=-f(a_2),
\end{aligned}
\]
and all other basis products are zero.  Hence $f$ is an automorphism.

Let $C_1$ be the subgroup consisting of the matrices
\[
 \begin{pmatrix}1&0&0\\0&1&0\\c&0&1\end{pmatrix},
 \qquad c\in\F,
\]
and let $C_2$ be the subgroup consisting of the matrices
\[
 \begin{pmatrix}1&0&0\\b&u&0\\0&0&1\end{pmatrix},
 \qquad b\in\F,
 \quad u\in\F^{\times}.
\]
For $c,d\in\F$,
\[
 \begin{pmatrix}1&0&0\\0&1&0\\c&0&1\end{pmatrix}
 \begin{pmatrix}1&0&0\\0&1&0\\d&0&1\end{pmatrix}
 =
 \begin{pmatrix}1&0&0\\0&1&0\\c+d&0&1\end{pmatrix},
\]
so $C_1\cong\F^{+}$.  Furthermore,
\[
 \begin{pmatrix}1&0&0\\0&1&0\\c&0&1\end{pmatrix}
 \begin{pmatrix}1&0&0\\b&u&0\\0&0&1\end{pmatrix}
 =
 \begin{pmatrix}1&0&0\\b&u&0\\c&0&1\end{pmatrix}
 =
 \begin{pmatrix}1&0&0\\b&u&0\\0&0&1\end{pmatrix}
 \begin{pmatrix}1&0&0\\0&1&0\\c&0&1\end{pmatrix}.
\]
Thus $C_1$ and $C_2$ commute.  Also $C_1\cap C_2=\{1\}$, and every matrix in~\eqref{eq:AutL7} has the displayed product form with
\[
 c=\alpha_3,\qquad b=\alpha_2,\qquad u=\beta_2.
\]
Hence $G=C_1C_2$ and
\[
 G=C_1\times C_2.
\]

Inside $C_2$, let
\[
 N=\left\{
 \begin{pmatrix}1&0&0\\b&1&0\\0&0&1\end{pmatrix}:b\in\F
 \right\}
\]
and
\[
 D=\left\{
 \begin{pmatrix}1&0&0\\0&u&0\\0&0&1\end{pmatrix}:u\in\F^{\times}
 \right\}.
\]
Then $N\cong\F^{+}$, $D\cong\F^{\times}$, $N\cap D=\{1\}$, and every element of $C_2$ is a product of an element of $N$ and an element of $D$.  Their conjugation is
\[
 \begin{pmatrix}1&0&0\\0&u&0\\0&0&1\end{pmatrix}
 \begin{pmatrix}1&0&0\\b&1&0\\0&0&1\end{pmatrix}
 \begin{pmatrix}1&0&0\\0&u^{-1}&0\\0&0&1\end{pmatrix}
 =
 \begin{pmatrix}1&0&0\\ub&1&0\\0&0&1\end{pmatrix}.
\]
Thus $N\trianglelefteq C_2$ and $C_2=N\rtimes D$, with the usual scalar action of $\F^{\times}$ on $\F^{+}$.  This proves the theorem.
\end{proof}

\subsection{The characteristic-two algebra \texorpdfstring{$L_8(\lambda)$}{L8(lambda)}}

\begin{theorem}\label{thm:L8}
Assume that $\operatorname{char}\F=2$, let $L=L_8(\lambda)$, and let $G=\Aut(L)$.  Then $\Xi(G)$ consists of all matrices
\begin{equation}\label{eq:AutL8}
\begin{pmatrix}
1&0&0\\
\alpha_2&\beta_2&0\\
\alpha_3&\alpha_2\beta_2&\beta_2^2
\end{pmatrix},
\qquad
\beta_2\in\F^{\times},\quad \alpha_3\in\F,
\end{equation}
where
\begin{equation}\label{eq:L8condition}
 \alpha_2^2=\lambda(\beta_2^2-1).
\end{equation}
\end{theorem}

\begin{proof}
Let $f\in G$.  Since
\[
 \Leib(L)=\F a_3,
 \qquad
 [L,L]=\F a_2\oplus\F a_3,
\]
we can write
\[
\begin{aligned}
 f(a_1)&=\alpha_1a_1+\alpha_2a_2+\alpha_3a_3,\\
 f(a_2)&=\beta_2a_2+\beta_3a_3,\\
 f(a_3)&=\gamma_3a_3,
\end{aligned}
\]
with $\alpha_1\beta_2\gamma_3\ne0$.
The equality $[a_1,a_2]=a_2$ gives
\[
\begin{aligned}
 \beta_2a_2+\beta_3a_3
 &=f([a_1,a_2])=[f(a_1),f(a_2)]\\
 &=[\alpha_1a_1+\alpha_2a_2+\alpha_3a_3,
    \beta_2a_2+\beta_3a_3]\\
 &=\alpha_1\beta_2a_2+\alpha_2\beta_2a_3.
\end{aligned}
\]
Therefore
\[
 \alpha_1=1,
 \qquad
 \beta_3=\alpha_2\beta_2.
\]
Using $[a_2,a_2]=a_3$, we obtain
\[
 \gamma_3a_3=f(a_3)=[f(a_2),f(a_2)]=\beta_2^2a_3,
\]
so $\gamma_3=\beta_2^2$.
We also use $[a_2,a_1]=a_2$.  Since $\operatorname{char}\F=2$,
\[
\begin{aligned}
 f(a_2)
 &=f([a_2,a_1])=[f(a_2),f(a_1)]\\
 &=[\beta_2a_2+\alpha_2\beta_2a_3,
    a_1+\alpha_2a_2+\alpha_3a_3]\\
 &=\beta_2a_2+\alpha_2\beta_2a_3,
\end{aligned}
\]
so this relation imposes no additional restriction.  Finally,
\[
\begin{aligned}
 \lambda\beta_2^2a_3
 &=f(\lambda a_3)=f([a_1,a_1])=[f(a_1),f(a_1)]\\
 &=[a_1+\alpha_2a_2+\alpha_3a_3,
    a_1+\alpha_2a_2+\alpha_3a_3]\\
 &=\lambda a_3+\alpha_2a_2+\alpha_2a_2+\alpha_2^2a_3
 =(\lambda+\alpha_2^2)a_3.
\end{aligned}
\]
Hence
\[
 \alpha_2^2=\lambda(\beta_2^2-1).
\]
This gives exactly the matrices in~\eqref{eq:AutL8}.

Conversely, take a nondegenerate matrix from~\eqref{eq:AutL8}.  Then
\[
\begin{aligned}
 [f(a_1),f(a_1)]
 &=\bigl(\lambda+\alpha_2^2\bigr)a_3
 =\lambda\beta_2^2a_3
 =\lambda f(a_3),\\
 [f(a_1),f(a_2)]
 &=\beta_2a_2+\alpha_2\beta_2a_3=f(a_2),\\
 [f(a_2),f(a_1)]
 &=\beta_2a_2+\alpha_2\beta_2a_3=f(a_2),\\
 [f(a_2),f(a_2)]&=\beta_2^2a_3=f(a_3).
\end{aligned}
\]
All products involving $a_3$ are zero.  Thus the displayed matrices are precisely the automorphisms.
\end{proof}

\begin{corollary}\label{cor:L8}
Under the assumptions of Theorem~\ref{thm:L8} the following assertions hold.
\begin{enumerate}
\item If $\lambda$ is not a square in $\F$, then $G\cong\F^{+}$.
\item If $\lambda$ is a square in $\F$, then $L_8(\lambda)\cong L_8(0)$, and $G$ is a semidirect product of a normal subgroup isomorphic to $\F^{+}$ and a subgroup isomorphic to $\F^{\times}$.  The element $u\in\F^{\times}$ acts on $z\in\F^{+}$ by $z\mapsto u^2z$.
\end{enumerate}
\end{corollary}

\begin{proof}
Since $\operatorname{char}\F=2$, we have
\[
 \beta_2^2-1=(\beta_2-1)^2.
\]
Suppose first that $\lambda$ is not a square in $\F$.  If $\beta_2\ne1$, then the equality
\[
 \alpha_2^2=\lambda(\beta_2^2-1)
\]
gives
\[
 \lambda=\left(\frac{\alpha_2}{\beta_2-1}\right)^2,
\]
a contradiction.  Hence $\beta_2=1$, and then $\alpha_2=0$.  Thus the automorphism group consists of the matrices
\[
 \begin{pmatrix}1&0&0\\0&1&0\\c&0&1\end{pmatrix},
 \qquad c\in\F,
\]
which form a group isomorphic to the additive group $\F^{+}$.

Now suppose that $\lambda$ is a square in $\F$.  By the parameter equivalence in Table~\ref{tab:status}, the algebra $L_8(\lambda)$ is isomorphic to $L_8(0)$, so it is enough to consider $\lambda=0$.  Theorem~\ref{thm:L8} gives
\[
 \Xi(G)=\left\{
 \begin{pmatrix}1&0&0\\0&u&0\\c&0&u^2\end{pmatrix}:
 c\in\F,\ u\in\F^{\times}
 \right\}.
\]
Let $C$ be the subgroup
\[
 \left\{
 \begin{pmatrix}1&0&0\\0&1&0\\c&0&1\end{pmatrix}:c\in\F
 \right\}
\]
and let $D$ be the subgroup
\[
 \left\{
 \begin{pmatrix}1&0&0\\0&u&0\\0&0&u^2\end{pmatrix}:u\in\F^{\times}
 \right\}.
\]
Then $C\cong\F^{+}$, $D\cong\F^{\times}$, $C\cap D=\{1\}$, and every element of $G$ is a product of an element of $C$ and an element of $D$.  Moreover,
\[
 \begin{pmatrix}1&0&0\\0&u&0\\0&0&u^2\end{pmatrix}
 \begin{pmatrix}1&0&0\\0&1&0\\c&0&1\end{pmatrix}
 \begin{pmatrix}1&0&0\\0&u^{-1}&0\\0&0&u^{-2}\end{pmatrix}
 =
 \begin{pmatrix}1&0&0\\0&1&0\\u^2c&0&1\end{pmatrix}.
\]
Therefore $C\trianglelefteq G$ and $G=C\rtimes D$, where $u\in\F^{\times}$ acts on $c\in\F^{+}$ by $c\mapsto u^2c$.  This also describes the automorphism group of every $L_8(\lambda)$ for which $\lambda$ is a square in $\F$.
\end{proof}

\subsection{The Leibniz algebra \texorpdfstring{$L_9$}{L9} in characteristic two}

For $\operatorname{char}\F\ne2$, the automorphism group of $L_9$ is contained in the result of Di Bartolo, La Rosa and Mancini~\cite{DiBartolo}.  We therefore consider only the missing case $\operatorname{char}\F=2$.

\begin{theorem}\label{thm:L9}
Assume $\operatorname{char}\F=2$, let $L=L_9$, and let $G=\Aut(L)$.  Then $\Xi(G)$ consists of all matrices
\begin{equation}\label{eq:AutL9}
\begin{pmatrix}
1&0&0\\
\alpha_2&\beta_2&0\\
\gamma_3-1&0&\gamma_3
\end{pmatrix},
\qquad
\alpha_2\in\F,\quad \beta_2,\gamma_3\in\F^{\times}.
\end{equation}
Moreover, $G$ is a direct product of a subgroup isomorphic to $\F^{\times}$ and a subgroup which is a semidirect product of $\F^{+}$ by $\F^{\times}$.
\end{theorem}

\begin{proof}
Let $f\in G$.  Here
\[
 \Leib(L)=[L,L]=\F a_3,
 \qquad
 \zeta(L)=\F a_2.
\]
Therefore
\[
\begin{aligned}
 f(a_1)&=\alpha_1a_1+\alpha_2a_2+\alpha_3a_3,\\
 f(a_2)&=\beta_2a_2,\\
 f(a_3)&=\gamma_3a_3,
\end{aligned}
\]
where $\alpha_1\beta_2\gamma_3\ne0$.
Using $[a_1,a_3]=a_3$, we obtain
\[
 \gamma_3a_3=f([a_1,a_3])=[f(a_1),f(a_3)]
 =\alpha_1\gamma_3a_3.
\]
Thus $\alpha_1=1$.  Next,
\[
\begin{aligned}
 \gamma_3a_3
 &=f([a_1,a_1])=[f(a_1),f(a_1)]\\
 &=[a_1+\alpha_2a_2+\alpha_3a_3,
    a_1+\alpha_2a_2+\alpha_3a_3]\\
 &=(1+\alpha_3)a_3.
\end{aligned}
\]
Hence $\alpha_3=\gamma_3-1$, and we obtain~\eqref{eq:AutL9}.

Conversely, for every nondegenerate matrix in~\eqref{eq:AutL9},
\[
 [f(a_1),f(a_1)]=\gamma_3a_3=f(a_3),
 \qquad
 [f(a_1),f(a_3)]=\gamma_3a_3=f(a_3),
\]
and every other product of basis images is zero.  Hence all displayed matrices define automorphisms.

Let $C_1$ be the subgroup consisting of the matrices
\[
 \begin{pmatrix}1&0&0\\b&u&0\\0&0&1\end{pmatrix},
 \qquad b\in\F,
 \quad u\in\F^{\times},
\]
and let $C_2$ be the subgroup consisting of the matrices
\[
 \begin{pmatrix}1&0&0\\0&1&0\\t-1&0&t\end{pmatrix},
 \qquad t\in\F^{\times}.
\]
The correspondence
\[
 \begin{pmatrix}1&0&0\\0&1&0\\t-1&0&t\end{pmatrix}\longmapsto t
\]
is an isomorphism $C_2\cong\F^{\times}$, because
\[
 \begin{pmatrix}1&0&0\\0&1&0\\t-1&0&t\end{pmatrix}
 \begin{pmatrix}1&0&0\\0&1&0\\s-1&0&s\end{pmatrix}
 =
 \begin{pmatrix}1&0&0\\0&1&0\\ts-1&0&ts\end{pmatrix}.
\]

Inside $C_1$, put
\[
 N=\left\{
 \begin{pmatrix}1&0&0\\b&1&0\\0&0&1\end{pmatrix}:b\in\F
 \right\},
 \qquad
 D=\left\{
 \begin{pmatrix}1&0&0\\0&u&0\\0&0&1\end{pmatrix}:u\in\F^{\times}
 \right\}.
\]
Then $N\cong\F^{+}$, $D\cong\F^{\times}$, $N\cap D=\{1\}$, and every element of $C_1$ belongs to $ND$.  Moreover,
\[
 \begin{pmatrix}1&0&0\\0&u&0\\0&0&1\end{pmatrix}
 \begin{pmatrix}1&0&0\\b&1&0\\0&0&1\end{pmatrix}
 \begin{pmatrix}1&0&0\\0&u^{-1}&0\\0&0&1\end{pmatrix}
 =
 \begin{pmatrix}1&0&0\\ub&1&0\\0&0&1\end{pmatrix}.
\]
Thus $N\trianglelefteq C_1$ and $C_1=N\rtimes D$.

Finally,
\[
 \begin{pmatrix}1&0&0\\b&u&0\\0&0&1\end{pmatrix}
 \begin{pmatrix}1&0&0\\0&1&0\\t-1&0&t\end{pmatrix}
 =
 \begin{pmatrix}1&0&0\\b&u&0\\t-1&0&t\end{pmatrix}
 =
 \begin{pmatrix}1&0&0\\0&1&0\\t-1&0&t\end{pmatrix}
 \begin{pmatrix}1&0&0\\b&u&0\\0&0&1\end{pmatrix}.
\]
Hence $C_1$ and $C_2$ commute.  Their intersection is trivial, and every matrix in~\eqref{eq:AutL9} is the displayed product with
\[
 b=\alpha_2,\qquad u=\beta_2,\qquad t=\gamma_3.
\]
Therefore $G=C_1C_2$ and
\[
 G=C_1\times C_2,
\]
as asserted.
\end{proof}

\subsection{The Leibniz algebra \texorpdfstring{$L_{10}(r)$}{L10(r)}}

\begin{theorem}\label{thm:L10}
Let $r\in\F^{\times}$, $L=L_{10}(r)$, and $G=\Aut(L)$.  Then $\Xi(G)$ consists of all matrices
\begin{equation}\label{eq:AutL10}
\begin{pmatrix}
1&0&0\\
\alpha_2&\beta_2&0\\[1mm]
\dfrac{\gamma_3-1}{r}&0&\gamma_3
\end{pmatrix},
\qquad
\alpha_2\in\F,\quad \beta_2,\gamma_3\in\F^{\times}.
\end{equation}
Moreover, $G$ is a direct product of a subgroup isomorphic to $\F^{\times}$ and a subgroup which is a semidirect product of $\F^{+}$ by $\F^{\times}$.  In particular, the abstract automorphism groups are isomorphic for all $r\ne0$.
\end{theorem}

\begin{proof}
Let $f\in G$.  We have
\[
 \Leib(L)=\F a_3,
 \qquad
 [L,L]=\F a_2\oplus\F a_3.
\]
Hence
\[
\begin{aligned}
 f(a_1)&=\alpha_1a_1+\alpha_2a_2+\alpha_3a_3,\\
 f(a_2)&=\beta_2a_2+\beta_3a_3,\\
 f(a_3)&=\gamma_3a_3,
\end{aligned}
\]
where $\alpha_1\beta_2\gamma_3\ne0$.
From $[a_2,a_1]=-a_2$ we get
\[
\begin{aligned}
 -\beta_2a_2-\beta_3a_3
 &=f([a_2,a_1])=[f(a_2),f(a_1)]\\
 &=-\alpha_1\beta_2a_2.
\end{aligned}
\]
Therefore
\[
 \alpha_1=1,
 \qquad
 \beta_3=0.
\]
Now
\[
\begin{aligned}
 \gamma_3a_3
 &=f([a_1,a_1])=[f(a_1),f(a_1)]\\
 &=[a_1+\alpha_2a_2+\alpha_3a_3,
    a_1+\alpha_2a_2+\alpha_3a_3]\\
 &=a_3+\alpha_2a_2-\alpha_2a_2+r\alpha_3a_3
 =(1+r\alpha_3)a_3.
\end{aligned}
\]
Thus
\[
 \alpha_3=\frac{\gamma_3-1}{r}.
\]
The relation $[a_1,a_2]=a_2$ gives
\[
 [f(a_1),f(a_2)]=\beta_2a_2=f(a_2),
\]
and the relation $[a_1,a_3]=ra_3$ gives
\[
 [f(a_1),f(a_3)]=r\gamma_3a_3=r f(a_3).
\]
Hence there are no additional restrictions and we obtain~\eqref{eq:AutL10}.

Conversely, every nondegenerate matrix in~\eqref{eq:AutL10} satisfies
\[
\begin{aligned}
 [f(a_1),f(a_1)]&=f(a_3),\\
 [f(a_1),f(a_2)]&=f(a_2),\\
 [f(a_2),f(a_1)]&=-f(a_2),\\
 [f(a_1),f(a_3)]&=r f(a_3),
\end{aligned}
\]
and all remaining basis products are zero.  Thus it defines an automorphism.

Let $C_1$ be the subgroup consisting of
\[
 \begin{pmatrix}1&0&0\\b&u&0\\0&0&1\end{pmatrix},
 \qquad b\in\F,
 \quad u\in\F^{\times},
\]
and let $C_2$ be the subgroup consisting of
\[
 \begin{pmatrix}1&0&0\\0&1&0\\(t-1)/r&0&t\end{pmatrix},
 \qquad t\in\F^{\times}.
\]
The correspondence
\[
 \begin{pmatrix}1&0&0\\0&1&0\\(t-1)/r&0&t\end{pmatrix}
 \longmapsto t
\]
is an isomorphism $C_2\cong\F^{\times}$, since
\[
 \begin{pmatrix}1&0&0\\0&1&0\\(t-1)/r&0&t\end{pmatrix}
 \begin{pmatrix}1&0&0\\0&1&0\\(s-1)/r&0&s\end{pmatrix}
 =
 \begin{pmatrix}1&0&0\\0&1&0\\(ts-1)/r&0&ts\end{pmatrix}.
\]

Inside $C_1$, put
\[
 N=\left\{
 \begin{pmatrix}1&0&0\\b&1&0\\0&0&1\end{pmatrix}:b\in\F
 \right\},
 \qquad
 D=\left\{
 \begin{pmatrix}1&0&0\\0&u&0\\0&0&1\end{pmatrix}:u\in\F^{\times}
 \right\}.
\]
Then $N\cong\F^{+}$, $D\cong\F^{\times}$, $N\cap D=\{1\}$, and every element of $C_1$ belongs to $ND$.  Moreover,
\[
 \begin{pmatrix}1&0&0\\0&u&0\\0&0&1\end{pmatrix}
 \begin{pmatrix}1&0&0\\b&1&0\\0&0&1\end{pmatrix}
 \begin{pmatrix}1&0&0\\0&u^{-1}&0\\0&0&1\end{pmatrix}
 =
 \begin{pmatrix}1&0&0\\ub&1&0\\0&0&1\end{pmatrix}.
\]
Thus $N\trianglelefteq C_1$ and $C_1=N\rtimes D$.

Finally,
\[
 \begin{pmatrix}1&0&0\\b&u&0\\0&0&1\end{pmatrix}
 \begin{pmatrix}1&0&0\\0&1&0\\(t-1)/r&0&t\end{pmatrix}
 =
 \begin{pmatrix}1&0&0\\b&u&0\\(t-1)/r&0&t\end{pmatrix}
 =
 \begin{pmatrix}1&0&0\\0&1&0\\(t-1)/r&0&t\end{pmatrix}
 \begin{pmatrix}1&0&0\\b&u&0\\0&0&1\end{pmatrix}.
\]
Hence $C_1$ and $C_2$ commute.  Their intersection is trivial, and every matrix in~\eqref{eq:AutL10} is the displayed product with
\[
 b=\alpha_2,\qquad u=\beta_2,\qquad t=\gamma_3.
\]
Therefore
\[
 G=C_1\times C_2.
\]
This decomposition does not depend on the value of $r\ne0$, so all groups $\Aut(L_{10}(r))$ are abstractly isomorphic.
\end{proof}

\subsection{The Leibniz algebra \texorpdfstring{$L_{11}$}{L11}}

\begin{theorem}\label{thm:L11}
Assume $\operatorname{char}\F\ne2$, let $L=L_{11}$, and let $G=\Aut(L)$.  Then $\Xi(G)$ consists of all matrices
\begin{equation}\label{eq:AutL11}
\begin{pmatrix}
1&0&0\\[1mm]
\alpha_2&\beta_2&0\\[1mm]
-\dfrac{\alpha_2^2}{2}&-\alpha_2\beta_2&\beta_2^2
\end{pmatrix},
\qquad
\alpha_2\in\F,\quad \beta_2\in\F^{\times}.
\end{equation}
Moreover, $G$ is a semidirect product of a normal subgroup isomorphic to $\F^{+}$ and a subgroup isomorphic to $\F^{\times}$, where the latter acts by ordinary scalar multiplication.
\end{theorem}

\begin{proof}
Let $f\in G$.  Since
\[
 \Leib(L)=\F a_3,
 \qquad
 [L,L]=\F a_2\oplus\F a_3,
\]
we write
\[
\begin{aligned}
 f(a_1)&=\alpha_1a_1+\alpha_2a_2+\alpha_3a_3,\\
 f(a_2)&=\beta_2a_2+\beta_3a_3,\\
 f(a_3)&=\gamma_3a_3,
\end{aligned}
\]
where $\alpha_1\beta_2\gamma_3\ne0$.
The relation $[a_2,a_2]=a_3$ gives
\[
 \gamma_3a_3=[f(a_2),f(a_2)]=\beta_2^2a_3,
\]
so
\[
 \gamma_3=\beta_2^2.
\]
Next,
\[
\begin{aligned}
 \beta_2a_2+\beta_3a_3
 &=f([a_1,a_2])=[f(a_1),f(a_2)]\\
 &=\alpha_1\beta_2a_2
   +(\alpha_2\beta_2+2\alpha_1\beta_3)a_3.
\end{aligned}
\]
Since $\beta_2\ne0$, we obtain
\[
 \alpha_1=1,
 \qquad
 \beta_3=-\alpha_2\beta_2.
\]
We also use $[a_2,a_1]=-a_2$.  We have
\[
\begin{aligned}
 -\beta_2a_2-\beta_3a_3
 &=f([a_2,a_1])=[f(a_2),f(a_1)]\\
 &=[\beta_2a_2+\beta_3a_3,
    a_1+\alpha_2a_2+\alpha_3a_3]\\
 &=-\beta_2a_2+\alpha_2\beta_2a_3.
\end{aligned}
\]
Thus again $\beta_3=-\alpha_2\beta_2$, so no additional condition is obtained.  Finally, because $[a_1,a_1]=0$,
\[
\begin{aligned}
0&=[f(a_1),f(a_1)]\\
 &=[a_1+\alpha_2a_2+\alpha_3a_3,
    a_1+\alpha_2a_2+\alpha_3a_3]\\
 &=(\alpha_2^2+2\alpha_3)a_3.
\end{aligned}
\]
Thus
\[
 \alpha_3=-\frac{\alpha_2^2}{2}.
\]
The relation $[a_1,a_3]=2a_3$ is then automatic, because
\[
 [f(a_1),f(a_3)]=2\beta_2^2a_3=2f(a_3).
\]
We have obtained~\eqref{eq:AutL11}.

Conversely, for a matrix in~\eqref{eq:AutL11} we have
\[
\begin{aligned}
 [f(a_2),f(a_2)]&=f(a_3),\\
 [f(a_1),f(a_2)]&=f(a_2),\\
 [f(a_2),f(a_1)]&=-f(a_2),\\
 [f(a_1),f(a_3)]&=2f(a_3),\\
 [f(a_1),f(a_1)]&=0.
\end{aligned}
\]
Hence every displayed nondegenerate matrix defines an automorphism.

Let $C$ be the subgroup consisting of the matrices obtained by putting $\beta_2=1$ in~\eqref{eq:AutL11}, and let $D$ be the subgroup obtained by putting $\alpha_2=0$.  Then
\[
 C=\left\{
 \begin{pmatrix}1&0&0\\b&1&0\\-b^2/2&-b&1\end{pmatrix}:b\in\F
 \right\}
 \cong\F^{+},
\]
and
\[
 D=\left\{
 \begin{pmatrix}1&0&0\\0&u&0\\0&0&u^2\end{pmatrix}:u\in\F^{\times}
 \right\}
 \cong\F^{\times}.
\]
Indeed, direct multiplication gives
\[
\begin{pmatrix}1&0&0\\b&1&0\\-b^2/2&-b&1\end{pmatrix}
\begin{pmatrix}1&0&0\\c&1&0\\-c^2/2&-c&1\end{pmatrix}
=
\begin{pmatrix}1&0&0\\b+c&1&0\\-(b+c)^2/2&-(b+c)&1\end{pmatrix}.
\]
Moreover,
\[
\begin{pmatrix}1&0&0\\0&u&0\\0&0&u^2\end{pmatrix}
\begin{pmatrix}1&0&0\\b&1&0\\-b^2/2&-b&1\end{pmatrix}
\begin{pmatrix}1&0&0\\0&u^{-1}&0\\0&0&u^{-2}\end{pmatrix}
=
\begin{pmatrix}1&0&0\\ub&1&0\\-(ub)^2/2&-ub&1\end{pmatrix}.
\]
Thus $C\trianglelefteq G$, $C\cap D=\{1\}$, and every element of $G$ belongs to $CD$.  Therefore $G=C\rtimes D$, with the usual scalar action of $\F^{\times}$ on $\F^{+}$.
\end{proof}

\section{Automorphism groups in the case \texorpdfstring{$\dim_{\F}\Leib(L)=2$}{dim Leib(L)=2}}

The automorphism groups of $L_{12}$, $L_{13}$, and $L_{15}$ have already been described in the papers cited in Table~\ref{tab:status}.  We do not repeat their proofs.  The case $L_{14}(1)$ was also treated in~\cite{AutNonNilpotent2026}.  Thus it remains to consider the other values of the parameter in $L_{14}(\rho)$ and to give a completely explicit three-dimensional form for the one-generated type-II algebra $L_{16}(\beta,\gamma)$.

\subsection{The Leibniz algebra \texorpdfstring{$L_{14}(\rho)$}{L14(rho)}}

\begin{theorem}\label{thm:L14}
Let $L=L_{14}(\rho)$, $\rho\in\F^{\times}$, and let $G=\Aut(L)$.

If $\rho=1$, the previously known result~\cite{AutNonNilpotent2026}, written in the present basis, is
\begin{equation}\label{eq:AutL14one}
\Xi(G)=
\left\{
\begin{pmatrix}
1&0&0\\
u-1&u&s\\v&v&t
\end{pmatrix}:
 ut-vs\ne0
\right\},
\end{equation}
and $G\cong GL_2(\F)$.

If $\rho\ne1$ and either $\operatorname{char}\F=2$ or $\rho\ne-1$, then
\begin{equation}\label{eq:AutL14generic}
\Xi(G)=
\left\{
\begin{pmatrix}
1&0&0\\u-1&u&0\\0&0&t
\end{pmatrix}:
 u,t\in\F^{\times}
\right\},
\end{equation}
and $G\cong\F^{\times}\times\F^{\times}$.

If $\operatorname{char}\F\ne2$ and $\rho=-1$, then, in addition to the matrices in~\eqref{eq:AutL14generic}, $\Xi(G)$ contains all matrices
\begin{equation}\label{eq:AutL14minus}
\begin{pmatrix}
-1&0&0\\
1&0&s\\
y&y&0
\end{pmatrix},
\qquad s,y\in\F^{\times}.
\end{equation}
In this case $G$ is a semidirect product of $\F^{\times}\times\F^{\times}$ by a group of order two, and the element of order two interchanges the two multiplicative factors.
\end{theorem}

\begin{proof}
The case $\rho=1$ was proved in~\cite{AutNonNilpotent2026}; we only rewrite the published result in the present basis.  The correspondence
\[
 \begin{pmatrix}u&s\\v&t\end{pmatrix}
 \longmapsto
 \begin{pmatrix}
 1&0&0\\u-1&u&s\\v&v&t
 \end{pmatrix}
\]
is a bijection from $GL_2(\F)$ onto the group in~\eqref{eq:AutL14one}.  If
\[
 A=\begin{pmatrix}u&s\\v&t\end{pmatrix},
 \qquad
 B=\begin{pmatrix}u'&s'\\v'&t'\end{pmatrix},
\]
then the product of the corresponding $3\times3$ matrices is
\[
 \begin{pmatrix}
 1&0&0\\
 uu'+sv'-1&uu'+sv'&us'+st'\\
 vu'+tv'&vu'+tv'&vs'+tt'
 \end{pmatrix},
\]
which is exactly the matrix corresponding to
\[
 AB=
 \begin{pmatrix}
 uu'+sv'&us'+st'\\
 vu'+tv'&vs'+tt'
 \end{pmatrix}.
\]
Thus the correspondence is a homomorphism.  Since it is already bijective, it is an isomorphism, and hence
\[
 \Aut(L_{14}(1))\cong GL_2(\F).
\]

Assume from now on that $\rho\ne1$, and let $f\in G$.  Since
\[
 \Leib(L)=[L,L]=\F a_2\oplus\F a_3
\]
is invariant under every automorphism, we have
\[
\begin{aligned}
 f(a_1)&=\alpha_1a_1+\alpha_2a_2+\alpha_3a_3,\\
 f(a_2)&=\beta_2a_2+\beta_3a_3,\\
 f(a_3)&=\gamma_2a_2+\gamma_3a_3,
\end{aligned}
\]
where
\[
 \alpha_1(\beta_2\gamma_3-\beta_3\gamma_2)\ne0.
\]
From $[a_1,a_1]=a_2$ we obtain
\[
\begin{aligned}
 \beta_2a_2+\beta_3a_3
 &= [f(a_1),f(a_1)]\\
 &=\alpha_1^2[a_1,a_1]
   +\alpha_1\alpha_2[a_1,a_2]
   +\alpha_1\alpha_3[a_1,a_3]\\
 &=\bigl(\alpha_1^2+\alpha_1\alpha_2\bigr)a_2
   +\alpha_1\rho\alpha_3a_3.
\end{aligned}
\]
Thus
\begin{equation}\label{eq:L14-1}
 \beta_2=\alpha_1^2+\alpha_1\alpha_2,
 \qquad
 \beta_3=\alpha_1\rho\alpha_3.
\end{equation}
Using $[a_1,a_2]=a_2$, we have
\[
\begin{aligned}
 \beta_2a_2+\beta_3a_3
 &=[f(a_1),f(a_2)]\\
 &=\alpha_1\beta_2a_2+\alpha_1\rho\beta_3a_3.
\end{aligned}
\]
Hence
\begin{equation}\label{eq:L14-2}
 (1-\alpha_1)\beta_2=0,
 \qquad
 (1-\alpha_1\rho)\beta_3=0.
\end{equation}
Finally, from $[a_1,a_3]=\rho a_3$ we obtain
\[
\begin{aligned}
 \rho\gamma_2a_2+\rho\gamma_3a_3
 &= [f(a_1),f(a_3)]\\
 &=\alpha_1\gamma_2a_2+\alpha_1\rho\gamma_3a_3,
\end{aligned}
\]
so
\begin{equation}\label{eq:L14-3}
 (\rho-\alpha_1)\gamma_2=0,
 \qquad
 (1-\alpha_1)\rho\gamma_3=0.
\end{equation}

Suppose first that $\alpha_1=1$.  Since $\rho\ne1$, equations~\eqref{eq:L14-2} and~\eqref{eq:L14-3} give
\[
 \beta_3=0,
 \qquad
 \gamma_2=0.
\]
By~\eqref{eq:L14-1},
\[
 \alpha_2=\beta_2-1,
 \qquad
 \alpha_3=0.
\]
The nondegeneracy condition gives $\beta_2\gamma_3\ne0$.  We obtain the matrices~\eqref{eq:AutL14generic}.

Suppose now that $\alpha_1\ne1$.  Equations~\eqref{eq:L14-2} and~\eqref{eq:L14-3} give
\[
 \beta_2=0,
 \qquad
 \gamma_3=0.
\]
Hence the nondegeneracy condition implies $\beta_3\gamma_2\ne0$.  Therefore the remaining equations in~\eqref{eq:L14-2} and~\eqref{eq:L14-3} yield
\[
 \alpha_1\rho=1,
 \qquad
 \alpha_1=\rho.
\]
Consequently $\rho^2=1$.  Since $\rho\ne1$, this is possible only when $\operatorname{char}\F\ne2$ and $\rho=-1$, and then $\alpha_1=-1$.  From~\eqref{eq:L14-1} we get
\[
 \alpha_2=1,
 \qquad
 \beta_3=\alpha_3.
\]
Thus the additional automorphisms are precisely the matrices~\eqref{eq:AutL14minus}.

We now verify the converse explicitly from the defining products.  For a matrix~\eqref{eq:AutL14generic},
\[
 [f(a_1),f(a_1)]=f(a_2),
 \qquad
 [f(a_1),f(a_2)]=f(a_2),
 \qquad
 [f(a_1),f(a_3)]=\rho f(a_3).
\]
Now let $\rho=-1$ and let $f$ have a matrix of the form~\eqref{eq:AutL14minus}.  Then
\[
 f(a_1)=-a_1+a_2+ya_3,\qquad
 f(a_2)=ya_3,\qquad
 f(a_3)=sa_2.
\]
Using
\[
 [a_1,a_1]=[a_1,a_2]=a_2,\qquad [a_1,a_3]=-a_3,
\]
we obtain
\[
\begin{aligned}
 [f(a_1),f(a_1)]
 &=a_2-a_2+ya_3=ya_3=f(a_2),\\
 [f(a_1),f(a_2)]
 &=-y[a_1,a_3]=ya_3=f(a_2),\\
 [f(a_1),f(a_3)]
 &=-s[a_1,a_2]=-sa_2=-f(a_3).
\end{aligned}
\]
Since $\F a_2\oplus\F a_3\le\zleft(L)$, all other products are zero.  Hence every displayed nondegenerate matrix is an automorphism.

If $\rho\ne1$ and either $\operatorname{char}\F=2$ or $\rho\ne-1$, then the group consists only of the matrices~\eqref{eq:AutL14generic}.  Multiplying two such matrices gives
\[
\begin{pmatrix}1&0&0\\u-1&u&0\\0&0&t\end{pmatrix}
\begin{pmatrix}1&0&0\\v-1&v&0\\0&0&s\end{pmatrix}
=
\begin{pmatrix}1&0&0\\uv-1&uv&0\\0&0&ts\end{pmatrix}.
\]
Therefore the correspondence of this matrix to $(u,t)$ is an isomorphism onto $\F^{\times}\times\F^{\times}$.

Let now $\operatorname{char}\F\ne2$ and $\rho=-1$.  Denote by $H$ the subgroup consisting of the matrices~\eqref{eq:AutL14generic}, and put
\[
 \tau=\begin{pmatrix}-1&0&0\\1&0&1\\1&1&0\end{pmatrix}.
\]
Then $\tau^2=1$.  Every matrix in~\eqref{eq:AutL14minus} belongs to $\tau H$, and every element of $G$ belongs either to $H$ or to $\tau H$.  The two sets do not intersect, because otherwise $\tau$ would belong to $H$.  Finally,
\[
 \tau
 \begin{pmatrix}1&0&0\\u-1&u&0\\0&0&t\end{pmatrix}
 \tau^{-1}
 =
 \begin{pmatrix}1&0&0\\t-1&t&0\\0&0&u\end{pmatrix}.
\]
Thus $H\trianglelefteq G$, $H\cong\F^{\times}\times\F^{\times}$, and $G$ is a semidirect product of $H$ by the subgroup $\langle\tau\rangle$ of order two; conjugation by $\tau$ interchanges the two multiplicative factors.
\end{proof}

\subsection{The Leibniz algebra \texorpdfstring{$L_{16}(\beta,\gamma)$}{L16(beta,gamma)}}

\begin{theorem}\label{thm:L16}
Let $L=L_{16}(\beta,\gamma)$, where $X^2-\gamma X-\beta$ is irreducible over $\F$, and let $G=\Aut(L)$.  In particular, $\beta\ne0$.

If $\gamma\ne0$ or $\operatorname{char}\F=2$, then $\Xi(G)$ consists of all matrices
\begin{equation}\label{eq:AutL16plus}
\begin{pmatrix}
1&0&0\\[1mm]
b-\dfrac{\gamma(a-1)}{\beta}&a&\beta b\\[3mm]
\dfrac{a-1}{\beta}&b&a+\gamma b
\end{pmatrix},
\qquad (a,b)\ne(0,0).
\end{equation}
If $E=\F(\theta)$, where $\theta^2=\gamma\theta+\beta$, then $G\cong E^{\times}$.

If $\operatorname{char}\F\ne2$ and $\gamma=0$, then, in addition to the matrices in~\eqref{eq:AutL16plus}, the group $\Xi(G)$ contains all matrices
\begin{equation}\label{eq:AutL16minus}
\begin{pmatrix}
-1&0&0\\[1mm]
-b&a&-\beta b\\[2mm]
\dfrac{1-a}{\beta}&b&-a
\end{pmatrix},
\qquad (a,b)\ne(0,0).
\end{equation}
In this case $G$ is a semidirect product of $E^{\times}$ by a group of order two; the nontrivial element acts on $E$ by $a+b\theta\mapsto a-b\theta$.
\end{theorem}

\begin{proof}
Let $f\in G$.  Put $K=\Leib(L)=\F a_2\oplus\F a_3$.  Since $K$ is invariant under every automorphism,
\[
\begin{aligned}
 f(a_1)&=\alpha_1a_1+\alpha_2a_2+\alpha_3a_3,\\
 f(a_2)&=\beta_2a_2+\beta_3a_3,\\
 f(a_3)&=\gamma_2a_2+\gamma_3a_3,
\end{aligned}
\]
where
\[
 \alpha_1(\beta_2\gamma_3-\beta_3\gamma_2)\ne0.
\]
We now apply the defining products one by one.

From $[a_1,a_1]=a_2$ we obtain
\[
\begin{aligned}
 \beta_2a_2+\beta_3a_3
 &=f(a_2)=[f(a_1),f(a_1)]\\
 &=\alpha_1^2[a_1,a_1]
   +\alpha_1\alpha_2[a_1,a_2]
   +\alpha_1\alpha_3[a_1,a_3]\\
 &=\alpha_1(\alpha_1+\beta\alpha_3)a_2
   +\alpha_1(\alpha_2+\gamma\alpha_3)a_3.
\end{aligned}
\]
Thus
\begin{equation}\label{eq:L16-1}
 \beta_2=\alpha_1(\alpha_1+\beta\alpha_3),
 \qquad
 \beta_3=\alpha_1(\alpha_2+\gamma\alpha_3).
\end{equation}
From $[a_1,a_2]=a_3$ we get
\[
\begin{aligned}
 \gamma_2a_2+\gamma_3a_3
 &=f(a_3)=[f(a_1),f(a_2)]\\
 &=\alpha_1\beta_2[a_1,a_2]
   +\alpha_1\beta_3[a_1,a_3]\\
 &=\alpha_1\beta\beta_3a_2
   +\alpha_1(\beta_2+\gamma\beta_3)a_3.
\end{aligned}
\]
Hence
\begin{equation}\label{eq:L16-2}
 \gamma_2=\alpha_1\beta\beta_3,
 \qquad
 \gamma_3=\alpha_1(\beta_2+\gamma\beta_3).
\end{equation}
Finally,
\[
 f(\beta a_2+\gamma a_3)=f([a_1,a_3])=[f(a_1),f(a_3)].
\]
Comparing the coefficients of $a_2$ and $a_3$, we obtain
\begin{align}
 \beta\beta_2+\gamma\gamma_2&=\alpha_1\beta\gamma_3, \label{eq:L16-3}\\
 \beta\beta_3+\gamma\gamma_3&=\alpha_1(\gamma_2+\gamma\gamma_3). \label{eq:L16-4}
\end{align}

We now extract the restriction on $\alpha_1$.  The four scalar equalities~\eqref{eq:L16-2}--\eqref{eq:L16-4} may be written together as
\[
 \begin{pmatrix}\beta_2&\gamma_2\\\beta_3&\gamma_3\end{pmatrix}
 \begin{pmatrix}0&\beta\\1&\gamma\end{pmatrix}
 =\alpha_1
 \begin{pmatrix}0&\beta\\1&\gamma\end{pmatrix}
 \begin{pmatrix}\beta_2&\gamma_2\\\beta_3&\gamma_3\end{pmatrix}.
\]
The first matrix is nondegenerate because $f$ induces an automorphism of
$K=\F a_2\oplus\F a_3$, and the second matrix has determinant $-\beta\ne0$.
Taking determinants of the displayed equality therefore gives
\[
 \det(f|_K)(-\beta)
 =\alpha_1^2(-\beta)\det(f|_K),
\]
and hence
\[
 \alpha_1^2=1.
\]
Multiplying the displayed matrix equality on the left by the inverse of the first matrix, we obtain
\[
 \begin{pmatrix}0&\beta\\1&\gamma\end{pmatrix}
 =
 \alpha_1
 \begin{pmatrix}\beta_2&\gamma_2\\\beta_3&\gamma_3\end{pmatrix}^{-1}
 \begin{pmatrix}0&\beta\\1&\gamma\end{pmatrix}
 \begin{pmatrix}\beta_2&\gamma_2\\\beta_3&\gamma_3\end{pmatrix}.
\]
The two matrices on the right which surround
$\begin{psmallmatrix}0&\beta\\1&\gamma\end{psmallmatrix}$
are inverse to one another.  Taking traces gives
\[
 \gamma=\alpha_1\gamma.
\]
Thus
\begin{equation}\label{eq:L16-alpha1}
 \alpha_1^2=1,
 \qquad
 (\alpha_1-1)\gamma=0.
\end{equation}

First suppose that $\alpha_1=1$.  Put
\[
 \beta_2=a,
 \qquad
 \beta_3=b.
\]
Then~\eqref{eq:L16-2} gives
\[
 \gamma_2=\beta b,
 \qquad
 \gamma_3=a+\gamma b.
\]
From~\eqref{eq:L16-1},
\[
 a=1+\beta\alpha_3,
 \qquad
 b=\alpha_2+\gamma\alpha_3,
\]
and therefore
\[
 \alpha_3=\frac{a-1}{\beta},
 \qquad
 \alpha_2=b-\frac{\gamma(a-1)}{\beta}.
\]
This gives the matrices~\eqref{eq:AutL16plus}.  Their determinant is
\[
 a^2+\gamma ab-\beta b^2.
\]
If this determinant were zero and $b\ne0$, then $a/b$ would be a root of
$X^2+\gamma X-\beta$, which is obtained from the irreducible polynomial
$X^2-\gamma X-\beta$ by replacing $X$ with $-X$.  This is impossible.  If $b=0$, then $(a,b)\ne(0,0)$ gives $a\ne0$.  Hence every matrix in~\eqref{eq:AutL16plus} is nondegenerate.

Conversely, for a matrix~\eqref{eq:AutL16plus}, direct calculation gives
\[
\begin{aligned}
 [f(a_1),f(a_1)]&=a a_2+b a_3=f(a_2),\\
 [f(a_1),f(a_2)]&=\beta b a_2+(a+\gamma b)a_3=f(a_3),\\
 [f(a_1),f(a_3)]
 &=\beta f(a_2)+\gamma f(a_3).
\end{aligned}
\]
All products with left factor $f(a_2)$ or $f(a_3)$ are zero because $K\le\zleft(L)$.  Thus these matrices are automorphisms.

By~\eqref{eq:L16-alpha1}, the case $\alpha_1=-1$ can occur only when $\operatorname{char}\F\ne2$ and $\gamma=0$.  In this case, put again $\beta_2=a$ and $\beta_3=b$.  Equalities~\eqref{eq:L16-2} give
\[
 \gamma_2=-\beta b,
 \qquad
 \gamma_3=-a,
\]
while~\eqref{eq:L16-1} gives
\[
 a=1-\beta\alpha_3,
 \qquad
 b=-\alpha_2.
\]
Hence
\[
 \alpha_2=-b,
 \qquad
 \alpha_3=\frac{1-a}{\beta},
\]
and we obtain the second family~\eqref{eq:AutL16minus}.  Its determinant is
\[
 a^2-\beta b^2.
\]
If $b\ne0$ and this determinant were zero, then $a/b$ would be a root of
$X^2-\beta$.  Since here $\gamma=0$, this is precisely the polynomial
$X^2-\gamma X-\beta$, which is irreducible by hypothesis.  This is impossible.
If $b=0$, then $(a,b)\ne(0,0)$ gives $a\ne0$.  Hence every matrix in~\eqref{eq:AutL16minus} is nondegenerate.

For such a matrix we have
\[
\begin{aligned}
 f(a_1)&=-a_1-ba_2+\frac{1-a}{\beta}a_3,\\
 f(a_2)&=aa_2+ba_3,\\
 f(a_3)&=-\beta ba_2-aa_3.
\end{aligned}
\]
Using $[a_1,a_1]=a_2$, $[a_1,a_2]=a_3$, and $[a_1,a_3]=\beta a_2$, we obtain
\[
\begin{aligned}
 [f(a_1),f(a_1)]
 &=aa_2+ba_3=f(a_2),\\
 [f(a_1),f(a_2)]
 &=-\beta ba_2-aa_3=f(a_3),\\
 [f(a_1),f(a_3)]
 &=\beta aa_2+\beta ba_3=\beta f(a_2).
\end{aligned}
\]
All products with left factor $f(a_2)$ or $f(a_3)$ are zero because
$K\le\zleft(L)$.  Therefore every matrix in~\eqref{eq:AutL16minus} defines an automorphism.

It remains to describe the abstract group.  Let $C$ be the subgroup formed by the matrices~\eqref{eq:AutL16plus}.  Let $E=\F(\theta)$, where
\[
 \theta^2=\gamma\theta+\beta.
\]
Define
\[
 \Phi:E^{\times}\longrightarrow C
\]
by sending $a+b\theta$ to the matrix~\eqref{eq:AutL16plus} with the same coefficients $a,b$.  We show directly that $\Phi$ is an isomorphism.  If
\[
 (a+b\theta)(c+d\theta)
 =(ac+\beta bd)+(ad+bc+\gamma bd)\theta,
\]
then multiplication of the corresponding matrices gives exactly the matrix~\eqref{eq:AutL16plus} with parameters
\[
 ac+\beta bd,
 \qquad
 ad+bc+\gamma bd.
\]
Therefore $\Phi$ is a homomorphism.  If $\Phi(a+b\theta)=I$, then the lower right $2\times2$ block gives $a=1$ and $b=0$, so the kernel is trivial.  Every matrix in $C$ occurs from some nonzero $a+b\theta$, and hence $\Phi$ is surjective.  Thus
\[
 C\cong E^{\times}.
\]

If $\gamma\ne0$ or $\operatorname{char}\F=2$, only the case $\alpha_1=1$ is possible, so $G=C$.  Suppose now that $\operatorname{char}\F\ne2$ and $\gamma=0$.  Put
\[
 \tau=\begin{pmatrix}-1&0&0\\0&1&0\\0&0&-1\end{pmatrix}.
\]
Then $\tau^2=1$, every matrix in~\eqref{eq:AutL16minus} belongs to $\tau C$, and every element of $G$ belongs either to $C$ or to $\tau C$.  Moreover, direct multiplication gives
\[
 \tau\,\Phi(a+b\theta)\,\tau^{-1}=\Phi(a-b\theta).
\]
Consequently $C\trianglelefteq G$, $C\cap\langle\tau\rangle=\{1\}$, and $G=C\langle\tau\rangle$.  Hence $G$ is a semidirect product of $E^{\times}$ by a group of order two; the nontrivial element acts by the field automorphism $a+b\theta\mapsto a-b\theta$.
\end{proof}

\section{Complete summary}

For convenience, we collect the results in one table.  In the rows corresponding to previously known cases, the published formulas are rewritten in the basis fixed in Table~\ref{tab:status}.  No new claim of priority is intended for those rows.

\begingroup
\small
\renewcommand{\arraystretch}{1.25}
\begin{longtable}{@{}P{0.09\textwidth}P{0.63\textwidth}P{0.22\textwidth}@{}}
\caption{Automorphism groups of all three-dimensional non-Lie Leibniz algebras in the refined list.}\label{tab:summary}\\
\toprule
Type & Description of $\Aut(L)$ & Source\\
\midrule
\endfirsthead
\toprule
Type & Description of $\Aut(L)$ & Source\\
\midrule
\endhead
\midrule
\multicolumn{3}{r}{\emph{Continued on next page}}\\
\endfoot
\bottomrule
\endlastfoot

$L_1$ &
$\left\{\begin{psmallmatrix}a&0&0\\ b&d&0\\ c&e&a^2\end{psmallmatrix}:a,d\in\F^{\times},\ b,c,e\in\F\right\}$ &
\cite{AutLowDims}\\

$L_2$ &
$\left\{\begin{psmallmatrix}a&0&0\\ b&a+b&0\\ c&d&a(a+b)\end{psmallmatrix}:a(a+b)\ne0\right\}$ &
\cite{AutLei3,AutNilpotent}\\

$L_3(\alpha)$ &
Formula~\eqref{eq:AutL3generic}; for $\alpha=1$ the additional matrices~\eqref{eq:AutL3exceptional} occur.  The normal subgroup $C_G(L/\F a_3)$ is isomorphic to $\F^{+}\times\F^{+}$; the corresponding complement is $\F^{\times}\times\F^{\times}$, with an additional factor of order two for $\alpha=1$. &
Theorem~\ref{thm:L3}\\

$L_4$ &
Formula~\eqref{eq:AutL4}; a semidirect product of $\F^{+}\times\F^{+}$ by $\F^{+}\times\F^{\times}$. &
Theorem~\ref{thm:L4}\\

$L_5(\beta)$ &
$\left\{\begin{psmallmatrix}a&-\varepsilon\beta b&0\\ b&\varepsilon a&0\\ c&d&a^2+\beta b^2\end{psmallmatrix}:(a,b)\ne(0,0),\ \varepsilon^2=1,\ c,d\in\F\right\}$ &
\cite{AutAnisotropic}\\

$L_6(\delta)$ &
Formula~\eqref{eq:AutL6}; a semidirect product of $\F^{+}\times\F^{+}$ by $E^{\times}$, where $E=\F(\theta)$ and $\theta^2-\theta+\delta=0$. &
Theorem~\ref{thm:L6}\\

$L_7$ &
Formula~\eqref{eq:AutL7}; a direct product of $\F^{+}$ and the semidirect product $\F^{+}\rtimes\F^{\times}$ with the usual scalar action. &
Theorem~\ref{thm:L7}\\

$L_8(\lambda)$ &
Formula~\eqref{eq:AutL8}.  If $\lambda$ is not a square in $\F$, then $\Aut(L)\cong\F^{+}$.  If $\lambda$ is a square in $\F$, then the group has a normal subgroup isomorphic to $\F^{+}$ and a complement isomorphic to $\F^{\times}$; conjugation sends $z$ to $u^2z$. &
Theorem~\ref{thm:L8}, Corollary~\ref{cor:L8}\\

$L_9$ &
Formula~\eqref{eq:AutL9}; a direct product of $\F^{\times}$ and the usual semidirect product $\F^{+}\rtimes\F^{\times}$. &
\cite{DiBartolo} if $\operatorname{char}\F\ne2$; Theorem~\ref{thm:L9} if $\operatorname{char}\F=2$\\

$L_{10}(r)$ &
Formula~\eqref{eq:AutL10}; a direct product of $\F^{\times}$ and the usual semidirect product $\F^{+}\rtimes\F^{\times}$. &
Theorem~\ref{thm:L10}\\

$L_{11}$ &
Formula~\eqref{eq:AutL11}; the semidirect product $\F^{+}\rtimes\F^{\times}$ with the usual scalar action. &
Theorem~\ref{thm:L11}\\

$L_{12}$ &
$\left\{\begin{psmallmatrix}a&0&0\\ b&a^2&0\\ c&ab&a^3\end{psmallmatrix}:a\in\F^{\times},\ b,c\in\F\right\}$ &
\cite{AutCyclic2022,AutCyclic,AutLowDims}\\

$L_{13}$ &
$\left\{\begin{psmallmatrix}1&0&0\\ b&1+b&0\\ c&b&1\end{psmallmatrix}:b\ne-1,\ c\in\F\right\}$ &
\cite{AutCyclic,AutNonNilpotent}; see Remark~\ref{rem:basis}\\

$L_{14}(\rho)$ &
$GL_2(\F)$ if $\rho=1$; $\F^{\times}\times\F^{\times}$ in the generic case; for $\operatorname{char}\F\ne2$ and $\rho=-1$, a semidirect product of $\F^{\times}\times\F^{\times}$ by a group of order two which interchanges the factors. &
\cite{AutNonNilpotent2026} for $\rho=1$; Theorem~\ref{thm:L14} otherwise\\

$L_{15}$ &
$\left\{\begin{psmallmatrix}1&0&0\\ a-1&a&0\\ b-a+1&b&a\end{psmallmatrix}:a\in\F^{\times},\ b\in\F\right\}$ &
\cite{AutCyclic,AutNonNilpotent2026}; see Remark~\ref{rem:basis}\\

$L_{16}(\beta,\gamma)$ &
Let $E=\F(\theta)$, $\theta^2=\gamma\theta+\beta$.  Then $\Aut(L)\cong E^{\times}$, except when $\operatorname{char}\F\ne2$ and $\gamma=0$, in which case $\Aut(L)$ is a semidirect product of $E^{\times}$ by a group of order two.  The matrices are~\eqref{eq:AutL16plus} and~\eqref{eq:AutL16minus}. &
Theorem~\ref{thm:L16}; compare \cite{AutCyclic2022,AutCyclic}\\

\end{longtable}
\endgroup

\begin{remark}
The table also shows that the automorphism group does not determine the isomorphism type of a three-dimensional Leibniz algebra.  For example, the parameter $r$ distinguishes the algebras $L_{10}(r)$, whereas the abstract automorphism groups obtained in Theorem~\ref{thm:L10} are isomorphic for all $r\ne0$.
\end{remark}

\end{document}